\documentclass[a4paper]{amsart}
\usepackage[active]{srcltx}
\usepackage[all]{xy}
\usepackage{subfig}
\usepackage{hyperref}
\usepackage{mathrsfs}

\usepackage{esint}
\usepackage{amsmath}
\usepackage{amssymb,latexsym}
\usepackage{mathrsfs}
\usepackage{graphics}
\usepackage{latexsym}
\usepackage{psfrag}
\usepackage{import}
\usepackage{verbatim}
\usepackage{graphicx}
\usepackage[usenames]{color}
\usepackage{pifont,marvosym}

\theoremstyle{plain}
\newtheorem{lemma}{Lemma}[section]
\newtheorem{theorem}[lemma]{Theorem}
\newtheorem{proposition}[lemma]{Proposition}
\newtheorem{corollary}[lemma]{Corollary}

\theoremstyle{definition}

\newtheorem{definition}[lemma]{Definition}
\newtheorem{remark}[lemma]{Remark}

\numberwithin{equation}{section}

\newcommand{\dom}{\textrm{Dom\,}}

\newcommand{\R}{\mathbb{R}}
\newcommand{\N}{\mathbb{N}}
\newcommand{\Q}{\mathbb{Q}}

\newcommand{\supp}{\text{\rm supp}}

\newcommand{\gr}{\textrm{graph}}

\newcommand{\diam}{\rm{diam\,}}
\newcommand{\haus}{\mathcal{H}}
\newcommand{\ve}{\varepsilon}

\newcommand{\erre}{\mathbb{R}}
\newcommand{\cI}{\mathcal{I}}

\newcommand{\f}{\varphi}

\newcommand{\T}{\mathcal{T}}
\renewcommand{\r}{\varrho}

\renewcommand{\L}{\mathcal{L}}
\newcommand{\RCD}{\mathsf{RCD}}

\newcommand{\CD}{\mathsf{CD}}
\newcommand{\Geo}{{\rm Geo}}

\newcommand{\mm}{\mathfrak m}
\newcommand{\qq}{\mathfrak q}
\newcommand{\QQ}{\mathfrak Q}

\newcommand{\sfd}{\mathsf d}
\newcommand{\Opt}{\mathrm{OptGeo}}
\newcommand{\bigslant}[2]{{\raisebox{.2em}{$#1$}\left/\raisebox{-.2em}{$#2$}\right.}}

\begin{document}

\title[Sharp isoperimetric inequalities in metric-measure spaces with  lower Ricci bounds] {Sharp and rigid isoperimetric inequalities in metric-measure spaces with  lower Ricci curvature bounds}
\author{Fabio Cavalletti}\thanks{F. Cavalletti: Universit\`a degli Studi di Pavia, Dipartimento di Matematica, email: fabio.cavalletti@unipv.it} 
\author{ Andrea Mondino} \thanks{A. Mondino: ETH-Z\"urich and  Universit\"at Z\"urich, Institut f\"ur Mathematik.  email: andrea.mondino@math.uzh.ch} 
%

%

\bibliographystyle{plain}

\begin{abstract}
We prove that if $(X,\sfd,\mm)$ is a metric measure space with $\mm(X)=1$ having (in a synthetic sense) Ricci curvature bounded from below by $K>0$ and dimension bounded above by $N\in [1,\infty)$, then the classic L\'evy-Gromov isoperimetric inequality (together with the recent sharpening counterparts proved in the smooth setting  by E. Milman for any $K\in \R$, $N\geq 1$ and upper diameter bounds) holds, i.e. the isoperimetric profile function of $(X,\sfd,\mm)$ is bounded from below by the isoperimetric profile of the model space. Moreover, if equality is attained for some volume $v \in (0,1)$ and $K$ is strictly positive, then the space must be a spherical suspension and in this case we completely classify the isoperimetric regions.  Finally we also establish the almost rigidity: if the  equality is almost attained for some volume $v \in (0,1)$ and $K$ is strictly positive, then the space  must be mGH close to  a  spherical suspension.
To our knowledge this is the first result about isoperimetric comparison for non smooth metric measure spaces satisfying Ricci curvature lower bounds.  Examples of spaces fitting our assumptions include measured Gromov-Hausdorff limits of Riemannian manifolds satisfying Ricci curvature lower bounds, Alexandrov spaces with curvature bounded from below, Finsler manifolds endowed with a strongly convex norm and satisfying  Ricci curvature lower bounds; the result seems new even in these celebrated classes of spaces.   
\end{abstract}

\maketitle


\section{Introduction}

\subsection{Isoperimetry}

The isoperimetric problem, having its roots in myths of more than 2000 years ago,  is one of the most classical and beautiful problems in mathematics. It amounts to answer the following natural questions:
\begin{enumerate}
\item Given a space $X$ what is the  minimal amount of  area  needed to enclose a fixed volume $v$?
\item   Does an optimal shape exist?
\item  In the affirmative case, can we describe the optimal shape?
\end{enumerate}

There are not many examples of spaces where the answer to all the three questions above is  known. If the space $X$ is the euclidean $N$-dimensional space $\R^N$ then it is well known that the only optimal shapes, called from now on isoperimetric regions, are  the round balls; if $X$ is the  round $N$-dimensional sphere ${\mathbb S}^N$ then the only isoperimetric regions are metric balls, etc.  To the  best of our knowledge, the spaces for which one can fully answer all the three questions above  either  have a \emph{very strong symmetry} or they are perturbations of spaces with a very  strong symmetry.  For an updated list of geometries admitting an isoperimetric description we refer to \cite[Appendix H]{EiMe}. Let us also mention that the isoperimetric problem has already been studied in presence of (mild) singularities of the space: mostly for conical manifolds \cite{MilRot,  MR} and polytopes \cite{MorPol}. The isoperimetric problem has been analyzed from several complementary points of view: for an overview of the more geometric aspects we refer to \cite{Oss, Rit, Ros}, for the approach via  geometric measure theory see for instance \cite{Mag, Mor}, for  the connections with convex and integral geometry see  \cite{BurZal}, for the point of view of optimal transport see \cite{FiMP, Vil}, for the recent quantitative forms see \cite{CL, FuMP}.

Besides the euclidean one, the most famous isoperimetric inequality is probably the L\'evy-Gromov  inequality \cite[Appendix C]{Gro}, which states that if $E$ is a (sufficiently regular) subset of a Riemannian manifold $(M^N,g)$ with dimension $N$ and Ricci bounded below  by $K>0$, then 
\begin{equation}\label{eq:LevyGromov}
\frac{|\partial E|}{|M|}\geq \frac{|\partial B|}{|S|},
\end{equation}
where $B$ is a spherical cap in the model sphere, i.e. the $N$-dimensional sphere with constant Ricci curvature equal to $K$,  and $|M|,|S|,|\partial E|, |\partial B|$  denote the appropriate $N$ or $N-1$ dimensional volume, and where $B$ is chosen so that
$|E|/|M|=|B|/|S|$. In other words, the L\'evy-Gromov isoperimetric inequality states that  isoperimetry in $(M,g)$ is at least as strong as in the model space $S$.
\\

Let us observe next that the isoperimetric problem makes sense in the larger class of metric measure spaces.  A metric measure space $(X,\sfd,\mm)$, m.m.s. for short,  is a metric space\footnote{during all the paper we will assume $(X,\sfd)$ to be complete, separable and proper} $(X,\sfd)$ endowed with a Borel probability measure $\mm$.  In the standard situation where  the metric space is a compact  Riemannian manifold, $\mm$ is nothing but the normalized volume measure. Notice that in the L\'evy-Gromov inequality \eqref{eq:LevyGromov} one considers exactly  this normalized  volume measure.
\\Regarding the m.m.s. setting, it is clear that the volume of a Borel set is replaced by its  $\mm$-measure, $\mm(E)$; the boundary area of the smooth framework instead can be replaced by the Minkowski content
\begin{equation}\label{def:MinkCont}
\mm^+(E):=\liminf_{\ve\downarrow 0} \frac{\mm(E^\ve)- \mm(E)}{\ve},
\end{equation}
where $E^{\ve}:=\{x \in X \,:\, \exists y \in  E \, \text{ such that } \, \sfd(x,y)< \ve \}$ is the $\ve$-neighborhood of $E$ with respect to the metric $\sfd$.
So the isoperimetric problem for a m.m.s. $(X,\sfd,\mm)$  amounts to finding the largest function $\cI_{(X,\sfd,\mm)}:[0,1]\to \R^+$ such that for every Borel subset $E\subset X$ it holds $\mm^+(E)\geq \cI_{(X,\sfd,\mm)}(\mm(E))$.
\\

The main goal of this paper is to prove that the L\'evy-Gromov isoperimetric inequality holds in the general framework of metric measure spaces. For the problem to make sense,  we also need  a notion  of ``Ricci curvature bounded below by $K$ and dimension bounded above by $N$''  for m.m.s..

\subsection{Ricci curvature lower bounds for metric measure spaces}
  The investigation about the topic began with the seminal papers of Lott-Villani \cite{lottvillani:metric}  and Sturm \cite{sturm:I, sturm:II}, though has been adapted considerably since the work of Bacher-Sturm \cite{BS10} and Ambrosio-Gigli-Savar\'e \cite{AGS11a, AGS11b}.  The crucial property of any such definition is the compatibility with the smooth Riemannian case and the stability with respect to  measured Gromov-Hausdorff convergence.
While a great deal of progress has been made in this latter general framework from both the analytic, geometric and structural points of view, see for instance  \cite{AGMR12, AGS, AGS11a, AGS11b,  AMS, AMSLocToGlob, BS10, cava:MongeRCD, cava:decomposition, cavasturm:MCP, EKS, GaMo,GigliSplitting, GMR2013, GMS2013,Ket, MN,R2011,Savare13,Vil}, the isoperimetric problem  has remained elusive.

The notion of  lower Ricci curvature bound on a general metric-measure space comes with two subtleties. The first is that of \emph{dimension}, and has been well understood since the work of Bakry-\'Emery  \cite{BakryEmery_diffusions} and Bakry-Ledoux \cite{BakryLedoux}:  in both the geometry and analysis of spaces with lower Ricci curvature bounds, it has become clear the correct statement is not that ``$X$ has Ricci curvature bounded from below by $K$'', but that ``$X$ has $N$-dimensional Ricci curvature bounded from below by $K$''. Such spaces are said to satisfy the $(K,N)$-\emph{Curvature Dimension} condition, $\CD(K,N)$ for short; a variant of this is that of \emph{reduced} curvature dimension bound, $\CD^*(K,N)$.  See \cite{BS10, BakryEmery_diffusions, BakryLedoux, sturm:II} and Section \ref{Ss:geom} for more on this.

The second subtle point is that the classical definition of a metric-measure space with lower Ricci curvature bounds allows for Finsler structures (see the last theorem in \cite{Vil}), which after the aforementioned works of Cheeger-Colding are known not to appear as limits of smooth manifolds with Ricci curvature lower bounds.  To address this issue, Ambrosio-Gigli-Savar\'e \cite{AGS11b} introduced a more restrictive condition which rules out Finsler geometries while retaining the stability properties under measured Gromov-Hausdorff convergence, see also \cite{AGMR12} for the present simplified axiomatization.  In short, one studies the Sobolev space $W^{1,2}(X)$ of functions on $X$.  This space is always a Banach space, and the imposed extra condition  is that $W^{1,2}(X)$ is a Hilbert space.  Equivalently, the Laplace operator on $X$ is linear.  The notion of a lower Ricci curvature bound compatible with this last Hilbertian condition is called  \emph{Riemannian Curvature Dimension} bound, $\RCD$ for short.  Refinements of this have led to the notion of $\RCD^*(K,N)$-spaces, which is the key object of study in this paper.  

\subsection{Main results}
Our main result is that the L\'evy-Gromov isoperimetric inequality holds for  m.m.s. satisfying $N$-Ricci curvature lower bounds:

\begin{theorem}[L\'evy-Gromov in $\RCD^*(K,N)$-spaces] \label{thm:LG}
Let  $(X,\sfd,\mm)$ be an $\RCD^*(K,N)$ space for some $N\in \N$ and $K>0$. Then for every Borel subset $E\subset X$ it holds
$$\mm^+(E)\geq  \frac{|\partial B|}{|S|},$$
where $B$ is a spherical cap in the model sphere (the $N$-dimensional sphere with constant Ricci curvature equal to $K$) chosen so that
$|B|/|S|=\mm(E)$.
\end{theorem}

Actually  Theorem \ref{thm:LG} will be just a particular case of the  more general Theorem \ref{thm:mainIsoComp} including any lower bound $K \in \R$ on the Ricci curvature and any upper bound $N \in [1,\infty)$ on the dimension. In order to state the result we need some model space to compare with: the same role that the round sphere played for the L\'evy-Gromov inequality. The model spaces for general $K,N$ have been discovered by E. Milman \cite{Mil} who extended the L\'evy-Gromov isoperimetric inequality to smooth manifolds with densities, i.e. smooth Riemannian manifold whose volume measure has been multiplied by a smooth non negative integrable density function. 
Milman detected a model isoperimetric profile  $\cI_{K,N,D}$ such that if a Riemannian manifold with density has diameter at most $D>0$, generalized Ricci curvature at least $K\in \R$ and generalized dimension at most $N\geq 1$ then the isoperimetric profile function of the weighted manifold is bounded below by   $\cI_{K,N,D}$. The main result of this paper is the non-smooth generalization of this statement:

\begin{theorem}[L\'evy-Gromov-Milman in $\RCD^*(K,N)$-spaces] \label{thm:mainIsoComp}
Let $(X,\sfd,\mm)$ be a metric measure space   with  $\mm(X)=1$ and  having diameter  $D\in (0,+\infty]$.  Assume is satisfies the   $\RCD^*(K,N)$ condition  for some $K\in \R, N \in (1,\infty)$ or $N=1, K\geq 0$. Then  for every Borel set $E\subset X$ it holds
$$\mm^+(E)\geq \cI_{K,N,D}(\mm(E)). $$
In other words it holds $\cI_{(X,\sfd,\mm)}(v)\geq \cI_{K,N,D}(v)$ for every $v \in [0,1]$.
\end{theorem}

\begin{remark}\label{rem:CD*nb}
Theorems \ref{thm:LG} and  \ref{thm:mainIsoComp} hold (and will be proved) in the more general framework of  essentially non branching $\CD_{loc}(K,N)$-spaces, but we decided to state them in this form so to give a unified presentation also with the rigidity statement below. The restriction $K\geq 0$ if $N=1$ is due to the fact that for $K<0$ and  $N=1$ the $\CD_{loc}(K,N)$ does not imply $\CD^*(K,N)$, see  Remark \ref{rk:CDCDs} for more details.
\end{remark}

The natural question is now rigidity: if for some $v \in (0,1)$ it holds $\cI_{(X,\sfd,\mm)}(v)= \cI_{K,N,\infty}(v)$,  does it imply that $X$ has a special structure? The answer is given by the following theorem. Before stating the  result let us observe that if $(X,\sfd,\mm)$ is an $\RCD^*(K,N)$ space for some $K>0$ then, 
called $\sfd':=\sqrt{\frac{K}{N-1}} \; \sfd$, we have that $(X,\sfd',\mm)$ is  $\RCD^*(N-1,N)$; 
in other words, if the Ricci lower bound is $K>0$ then up to scaling we can assume it is actually equal to $N-1$.

\begin{theorem}\label{thm:Rigidity}
Let $(X,\sfd,\mm)$ be an  $\RCD^*(N-1,N)$ space for some  $N \in [2,\infty)$, with  $\mm(X)=1$. 
Assume that there exists $\bar{v} \in (0,1)$ such that $\cI_{(X,\sfd,\mm)}(\bar{v})=\cI_{N-1,N,\infty}(\bar{v})$. 
\medskip

Then $(X,\sfd,\mm)$ is a spherical suspension:  there exists
an $\RCD^*(N-2,N-1)$ space $(Y,\sfd_{Y}, \mm_{Y})$ with $\mm_{Y}(Y)=1$ such that  $X$ is isomorphic as metric measure space to $[0,\pi] \times^{N-1}_{\sin} Y$.

Moreover, in this case, the following hold:
\begin{itemize}
\item[$i)$]  For every $v\in [0,1]$ it holds  $\cI_{(X,\sfd,\mm)}(v)=\cI_{N-1,N,\infty}(v)$. 
\item[$ii)$] For every $v\in [0,1]$ there exists a Borel subset $A \subset X$ with $\mm(A)=v$ such that 
$$
\mm^+(A)=\cI_{(X,\sfd,\mm)}(v)=\cI_{N-1,N,\infty}(v).
$$
\item[$iii)$]  If $\mm(A)\in (0,1)$ then  $\mm^+(A)=\cI_{(X,\sfd,\mm)}(v)=\cI_{N-1,N,\infty}(v)$ if and only if
$$
\bar{A}=\{(t,y)\in [0,\pi] \times^{N-1}_{\sin} Y \, :\, t \in [0,r_v] \} \quad \text{or} \quad   \bar{A}=\{(t,y)\in [0,\pi] \times^{N-1}_{\sin} Y \, :\, t \in [\pi-r_v, \pi] \},
$$
where $\bar{A}$ is the closure of $A$ and  $r_v\in (0,\pi)$ is chosen so 
that $\int_{[0,r_v]} c_N (\sin(t))^{N-1} dt =v$, $c_N$ being given by $c_N^{-1}:= \int_{[0,\pi]}  (\sin(t))^{N-1} dt$.
\end{itemize}
\end{theorem}

A last question we address here is the almost rigidity:  if $(X,\sfd,\mm)$ is an $\RCD^*(K,N)$ space such that    $\cI_{(X,\sfd,\mm)}(v)$ is close to $\cI_{K,N,\infty}(v)$ for some $v \in (0,1)$,  does this force  $X$ to be close to a spherical suspension? Let us mention that variants of  this problem were addressed  for   smooth Riemannian $N$-manifolds satisfying Ricci $\geq N-1$: Croke \cite{Croke} proved that the ratio between the profile of the manifold over the profile of the sphere is uniformly bounded from below by a constant which is strictly more than 1 as soon as the compact manifold is not isometric to the canonical $N$-sphere; this has been quantitatively estimated in \cite{BBG} where B\'erard, Besson, and Gallot gave explicit expressions of the infimum of the ratio in terms of the Ricci curvature and the diameter; finally  Bayle \cite{Bayle} proved that if the isoperimetric profile is close in the  \emph{uniform norm} to the one of the $N$-sphere then the diameter is almost maximal; combining this fact with the Maximal Diameter Theorem for limit spaces proved by Cheeger-Colding \cite{CC96},  one gets that the manifold must be close to a spherical suspension.    

The next Theorem  \ref{thm:AlmRig} together with its Corollary \ref{cor:AlmRig} extend the above results in two ways: first of all we assume closeness just for \emph{some} $v\in (0,1)$ and not  uniform closeness for every $v\in [0,1]$, second  we give a complete answer in  the larger  class of $\RCD^*(K,N)$ spaces.

\begin{theorem}[Almost equality in L\'evy-Gromov implies almost maximal diameter]\label{thm:AlmRig}
For every  $N>1$, $v \in (0,1)$, $\ve>0$ there exists $\bar{\delta}=\bar{\delta}(N,v,\ve)>0$ such that the following holds. For every $\delta\in [0, \bar{\delta}]$,  if $(X,\sfd,\mm)$ is an $\RCD^*(N-1-\delta,N+\delta)$ space satisfying 
$$\cI_{(X,\sfd,\mm)}(v)\leq \cI_{N-1,N,\infty}(v)+\delta, $$
Then $\diam((X,\sfd)) \geq \pi-\ve$.
\end{theorem}

The following corollary is a consequence of the Maximal Diameter Theorem \cite{Ket}, and of  the compactness/stability of the class of $\RCD^*(K,N)$ spaces, for some fixed $K>0$ and $N>1$,  with respect to the measured Gromov-Hausdorff convergence. Recall also that the measured Gromov Hausdorff convergence restricted to (isomorphism classes of) $\RCD^*(K,N)$ spaces   is metrizable (for more details see  Subsection \ref{SS:mGHConv}). 

\begin{corollary}[Almost equality in L\'evy-Gromov implies mGH-closeness to a spherical suspension] \label{cor:AlmRig}
For every $N\in [2, \infty) $, $v \in (0,1)$, $\ve>0$ there exists $\bar{\delta}=\bar{\delta}(N,v,\ve)>0$ such that the following hold. For every  $\delta \in [0, \bar{\delta}]$, if  $(X,\sfd,\mm)$ is an $\RCD^*(N-1-\delta,N+\delta)$ space satisfying 
$$\cI_{(X,\sfd,\mm)}(v)\leq \cI_{N-1,N,\infty}(v)+\delta, $$
then  there exists an $\RCD^*(N-2,N-1)$ space $(Y, \sfd_Y, \mm_Y)$ with $\mm_Y(Y)=1$ such that 
$$\sfd_{mGH}(X, [0,\pi] \times_{\sin}^{N-1} Y) \leq \ve. $$
\end{corollary}

\begin{remark}[Notable examples of spaces fitting in the assumptions of the main theorems]
The class of $\RCD^*(K,N)$ spaces include many remarkable family of spaces, among them:
\begin{itemize}
\item \emph{Measured Gromov Hausdorff limits of Riemannian $N$-dimensional manifolds  satisfying Ricci $\geq K$}. Despite the fine structural  properties of   such spaces discovered in a series of works by Cheeger-Colding \cite{CC1,CC2,CC3} and Colding-Naber \cite{CN}, the validity of the L\'evy-Gromov isoperimetric inequality (and the above generalizations and rigidity statements) has remained elusive.  We believe this is one of the most striking applications of our results.   For Ricci limit spaces let us  also mention the recent work by Honda \cite{Honda} where  a lower bound on the Cheeger constant is given, thanks to a stability argument on the first eigenvalue of the $p$-Laplacian for $p=1$.
\item \emph{Alexandrov spaces with curvature bounded from below}. 
Petrunin \cite{PLSV} proved that the lower curvature bound in the sense of comparison angles is compatible with the optimal transport type lower bound on the  Ricci curvature given by Lott-Sturm-Villani (see also \cite{zhangzhu}).  Moreover it is well known that the Laplace operator on an Alexandrov space is linear. It follows that Alexandrov spaces with curvature bounded from below are examples of $\RCD^*(K,N)$ and therefore our results apply as well. Let us note that in the framework of Alexandrov spaces the best result regarding isoperimetry is a sketch of a proof by Petrunin \cite{Pet} of the L\'evy-Gromov inequality for Alexandrov spaces with (sectional) curvature bounded below by 1. 
\end{itemize}
A last class of spaces where  Theorems \ref{thm:LG} and  \ref{thm:mainIsoComp} apply is the one of smooth  Finsler manifolds where the norm on the tangent spaces is strongly convex, and which satisfy lower Ricci curvature bounds. More precisely we consider a $C^{\infty}$-manifold  $M$, endowed with a function $F:TM\to[0,\infty]$ such that $F|_{TM\setminus \{0\}}$ is $C^{\infty}$ and  for each $p \in M$ it holds that $F_p:=T_pM\to [0,\infty]$ is a  strongly-convex norm, i.e.
$$g^p_{ij}(v):=\frac{\partial^2 (F_p^2)}{\partial v^i \partial v^j}(v) \quad \text{is a positive definite matrix at every } v \in T_pM\setminus\{0\}. $$
Under these conditions, it is known that one can write the  geodesic equations and geodesics do not branch; in other words these spaces are non-branching. 
We also assume $(M,F)$ to be geodesically complete and endowed with a $C^{\infty}$ probability measure $\mm$ in a such a way that the associated m.m.s. $(X,F,\mm)$ satisfies the $\CD^*(K,N)$ condition. This class of spaces has been investigated by Ohta \cite{Ohta} who established the equivalence between the Curvature Dimension condition and a Finsler-version of Bakry-Emery $N$-Ricci tensor bounded from below.  Recalling Remark \ref{rem:CD*nb},  these spaces  fit in the assumptions of  Theorems \ref{thm:LG}-\ref{thm:mainIsoComp}, and  to our knowledge the L\'evy-Gromov inequality (and its generalizations) is new also in this framework. \hfill$\Box$
\end{remark}

\subsection{Outline of the argument}
The main reason why the L\'evy-Gromov type inequalities have remained elusive in non smooth metric measure spaces is because the known proofs heavily rely on the existence and  sharp regularity properties of isoperimetric regions ensured by Geometric Measure Theory (see for instance \cite{Gro,Mag,Mor}).  Clearly such tools are available if the ambient space is a smooth Riemannian manifold (possibly endowed with a weighted measure, with smooth and strictly positive weight), but are out of disposal for general metric measure spaces.  

In order to overcome this huge difficulty we have been inspired by a paper of Klartag \cite{klartag} where the author gave a proof of the L\'evy-Gromov isoperimetric inequality still  in the framework of smooth Riemannian manifolds, but via an optimal transportation argument involving $L^1$-transportation and ideas of convex geometry. In particular he used a  localization technique, having its roots  in a work of   Payne-Weinberger \cite{PW} and developed by Gromov-Milman \cite{GrMi}, Lov\'asz-Simonovits \cite{LoSi} and Kannan-Lov\'asz-Simonovits \cite{KaLoSi}, which consists in reducing an $n$-dimensional problem, via tools of convex geometry,  to one-dimensional problems that one can handle.
\\

Let us stress that even if the approach by Klartag \cite{klartag} does not rely on the regularity of the isoperimetric region, it still heavily  makes use of the smoothness of the ambient space in order to establish sharp properties of the geodesics in terms of Jacobi fields and estimates on the  second fundamental forms of suitable level sets, all objects that are still not enough understood in general m.m.s. in order to repeat the same arguments.
\\

To overcome this difficulty we use the structural properties of geodesics and of $L^1$-optimal transport implied by the $\CD^*(K,N)$ condition.  Such results have their roots in previous works of Bianchini-Cavalletti \cite{biacava:streconv}
 and the first author \cite{cava:MongeRCD, cava:decomposition},  and will be developed in Sections  \ref{S:dmonotone} and \ref{sec:ConditionalMeasures}.  The first key point is to understand the structure of $\sfd$-monotone sets, in particular we will prove that under the curvature condition one can decompose the space, up to a set of measure zero, in equivalence classes called rays where the $L^1$-transport is performed (see Theorem \ref{T:summary}). A second key point, which is the technical novelty of the present work with respect to the aforementioned papers \cite{biacava:streconv, cava:MongeRCD, cava:decomposition},  is that on almost every  ray the conditional measure satisfies a precise curvature inequality (see Theorem \ref{T:CDKN-1}).
 This last technical novelty is exactly the key to reduce the problem on the original m.m.s. to a  one dimensional problem.
 
 This reduction is performed in Section \ref{S:Local} where we adapt to the non-smooth framework methods of convex geometry developed in the aforementioned papers \cite{klartag,GrMi,LoSi,KaLoSi}. The main result of the section is Theorem \ref{T:localize} asserting that if $f$ is an $L^1$-function with null mean value on an $\RCD^*(K,N)$-space $(X,\sfd,\mm)$, then we can disintegrate the measure along $\sfd$-monotone rays on which the induced measure satisfies a curvature condition  and such that    the function along a.e. ray still has null mean value. 
 
 In the final Section \ref{S:Isop} we apply these techniques to prove the main theorems. The idea is to use  Theorem \ref{T:localize}  to reduce the study of isoperimetry for Borel subset of $X$, to the study of isoperimetry for Borel subsets of the real line endowed with a measure satisfying suitable curvature condition. A tricky point is that the measure on the real line is a priori non smooth, while the results of Milman \cite{Mil} regarding isoperimetric comparison for manifolds with density are stated for smooth densities. This point is fixed by a non-linear regularization process which permits to regularize the densities maintaining the convexity conditions equivalent to the lower  Ricci curvature bounds (see Lemma \ref{lem:approxh} and Theorem \ref{thm:I=Is}). 
 
The proof of the L\'evy-Gromov  inequality (and its generalization) will then consist in  combining  the  dimension reduction argument, the regularization process, 
and the Isoperimetric Comparison proved by Milman \cite{Mil}  for smooth manifolds with densities.
The (resp. almost) rigidity statement will follow by observing that if the space  has (resp. almost) 
minimal isoperimetric profile then it must have (resp. almost) maximal diameter, and so the Maximal Diameter Theorem proved by Ketterer \cite{Ket} 
(resp. combined with the compactness/stability properties of the class of $\RCD^*(K,N)$ spaces) will force the space to be (resp. almost) a spherical suspension.  
To obtain the complete characterization of isoperimetric regions we  will perform  a careful analysis of the disintegration of the space induced by an optimal set.
    
\subsection{Future developments}
In the present paper we decided to focus on the isoperimetric problem, due to its relevance in many fields of Mathematics. In the following \cite{CM2} we will employ the techniques developed in this paper to prove functional inequalities like spectral gap, Poincar\'e and log-Sobolev inequalities, the Payne-Weinberger/Yang-Zhong inequality, among others. Some of these inequalities are consequences of the four functions theorem of Kannan, Lov\'asz and Simonovits.

\section*{Acknowledgements}
The authors wish to thank Emanuel Milman for having drawn their attention to the recent paper by Klartag \cite{klartag} and  the reviewers, whose detailed comments led to an improvement of the manuscript. They also wish to thank the Hausdorff center of Mathematics of Bonn, where most of the work has been developed,  for the excellent working conditions and the stimulating atmosphere during the trimester program ``Optimal Transport'' in Spring 2015. \\The second author gratefully acknowledges the support of the ETH-fellowship.

\section{Prerequisites}

In what follows we say that a triple $(X,\sfd, \mm)$ is a metric measure space, m.m.s. for short, 
if $(X, \sfd)$ is a complete and separable metric space and $\mm$ is positive Radon measure over $X$. 
For this paper we will only be concerned with m.m.s. with $\mm$ probability measure, that is $\mm(X) =1$.
The space of all Borel probability measures over $X$ will be denoted by $\mathcal{P}(X)$.

A metric space is a geodesic space if and only if for each $x,y \in X$ 
there exists $\gamma \in \Geo(X)$ so that $\gamma_{0} =x, \gamma_{1} = y$, with
$$
\Geo(X) : = \{ \gamma \in C([0,1], X):  \sfd(\gamma_{s},\gamma_{t}) = |s-t| \sfd(\gamma_{0},\gamma_{1}), \text{ for every } s,t \in [0,1] \}.
$$
Recall that for complete geodesic spaces local compactness is equivalent to properness (a metric space is proper if every closed ball is compact).
We directly assume the ambient space $(X,\sfd)$ to be proper. Hence from now on we assume the following:
the ambient metric space $(X, \sfd)$ is geodesic, complete, separable and proper and $\mm(X) = 1$.

\medskip

We denote with $\mathcal{P}_{2}(X)$ the space of probability measures with finite second moment  endowed with the $L^{2}$-Wasserstein distance  $W_{2}$ defined as follows:  for $\mu_0,\mu_1 \in \mathcal{P}_{2}(X)$ we set
\begin{equation}\label{eq:Wdef}
  W_2^2(\mu_0,\mu_1) = \inf_{ \pi} \int_{X\times X} \sfd^2(x,y) \, \pi(dxdy),
\end{equation}
where the infimum is taken over all $\pi \in \mathcal{P}(X \times X)$ with $\mu_0$ and $\mu_1$ as the first and the second marginal.
Assuming the space $(X,\sfd)$ to be geodesic, also the space $(\mathcal{P}_2(X), W_2)$ is geodesic. 

Any geodesic $(\mu_t)_{t \in [0,1]}$ in $(\mathcal{P}_2(X), W_2)$  can be lifted to a measure $\nu \in {\mathcal {P}}(\Geo(X))$, 
so that $({\rm e}_t)_\sharp \, \nu = \mu_t$ for all $t \in [0,1]$. 
Here for any $t\in [0,1]$,  ${\rm e}_{t}$ denotes the evaluation map: 
$$
  {\rm e}_{t} : \Geo(X) \to X, \qquad {\rm e}_{t}(\gamma) : = \gamma_{t}.
$$

Given $\mu_{0},\mu_{1} \in \mathcal{P}_{2}(X)$, we denote by 
$\Opt(\mu_{0},\mu_{1})$ the space of all $\nu \in \mathcal{P}(\Geo(X))$ for which $({\rm e}_0,{\rm e}_1)_\sharp\, \nu$ 
realizes the minimum in \eqref{eq:Wdef}. If $(X,\sfd)$ is geodesic, then the set  $\Opt(\mu_{0},\mu_{1})$ is non-empty for any $\mu_0,\mu_1\in \mathcal{P}_2(X)$.
It is worth also introducing the subspace of $\mathcal{P}_{2}(X)$
formed by all those measures absolutely continuous with respect with $\mm$: it is denoted by $\mathcal{P}_{2}(X,\sfd,\mm)$.


\subsection{Geometry of metric measure spaces}\label{Ss:geom}
Here we briefly recall the synthetic notions of lower Ricci curvature bounds, for more detail we refer to  \cite{BS10,lottvillani:metric,sturm:I, sturm:II, Vil}.

%
%

In order to formulate the curvature properties for $(X,\sfd,\mm)$ we introduce the following distortion coefficients: given two numbers $K,N\in \erre$ with $N\geq0$, we set for $(t,\theta) \in[0,1] \times \erre_{+}$, 
\begin{equation}\label{E:sigma}
\sigma_{K,N}^{(t)}(\theta):= 
\begin{cases}
\infty, & \textrm{if}\ K\theta^{2} \geq N\pi^{2}, \crcr
\displaystyle  \frac{\sin(t\theta\sqrt{K/N})}{\sin(\theta\sqrt{K/N})} & \textrm{if}\ 0< K\theta^{2} <  N\pi^{2}, \crcr
t & \textrm{if}\ K \theta^{2}<0 \ \textrm{and}\ N=0, \ \textrm{or  if}\ K \theta^{2}=0,  \crcr
\displaystyle   \frac{\sinh(t\theta\sqrt{-K/N})}{\sinh(\theta\sqrt{-K/N})} & \textrm{if}\ K\theta^{2} \leq 0 \ \textrm{and}\ N>0.
\end{cases}
\end{equation}

We also set, for $N\geq 1, K \in \R$ and $(t,\theta) \in[0,1] \times \erre_{+}$
\begin{equation} \label{E:tau}
\tau_{K,N}^{(t)}(\theta): = t^{1/N} \sigma_{K,N-1}^{(t)}(\theta)^{(N-1)/N}.
\end{equation}

%
%
%
%
%
%
%
%
%

As we will consider only the case of essentially non-branching spaces, we recall the following definition. 
\begin{definition}\label{D:essnonbranch}
A metric measure space $(X,\sfd, \mm)$ is \emph{essentially non-branching} if and only if for any $\mu_{0},\mu_{1} \in \mathcal{P}_{2}(X)$,
with $\mu_{0}$ absolutely continuous with respect to $\mm$, any element of $\Opt(\mu_{0},\mu_{1})$ is concentrated on a set of non-branching geodesics.
\end{definition}

A set $F \subset \Geo(X)$ is a set of non-branching geodesics if and only if for any $\gamma^{1},\gamma^{2} \in F$, it holds:
$$
\exists \;  \bar t\in (0,1) \text{ such that } \ \forall t \in [0, \bar t\,] \quad  \gamma_{ t}^{1} = \gamma_{t}^{2}   
\quad 
\Longrightarrow 
\quad 
\gamma^{1}_{s} = \gamma^{2}_{s}, \quad \forall s \in [0,1].
$$

\begin{definition}[$\CD$ condition]\label{D:CD}
An essentially non-branching m.m.s. $(X,\sfd,\mm)$ verifies $\mathsf{CD}(K,N)$  if and only if for each pair 
$\mu_{0}, \mu_{1} \in \mathcal{P}_{2}(X,\sfd,\mm)$ there exists $\nu \in \Opt(\mu_{0},\mu_{1})$ such that
\begin{equation}\label{E:CD}
\r_{t}^{-1/N} (\gamma_{t}) \geq  \tau_{K,N}^{(1-t)}(\sfd( \gamma_{0}, \gamma_{1}))\r_{0}^{-1/N}(\gamma_{0}) 
 + \tau_{K,N}^{(t)}(\sfd(\gamma_{0},\gamma_{1}))\r_{1}^{-1/N}(\gamma_{1}), \qquad \nu\text{-a.e.} \, \gamma \in \Geo(X),
\end{equation}
for all $t \in [0,1]$, where $({\rm e}_{t})_\sharp \, \nu = \r_{t} \mm$.
\end{definition}

For the general definition of $\CD(K,N)$ see \cite{lottvillani:metric, sturm:I, sturm:II}. It is worth recalling that if $(M,g)$ is a Riemannian manifold of dimension $n$ and 
$h \in C^{2}(M)$ with $h > 0$, then the m.m.s. $(M,g,h \, vol)$ verifies $\CD(K,N)$ with $N\geq n$ if and only if  (see Theorem 1.7 of \cite{sturm:II})
$$
Ric_{g,h,N} \geq  K g, \qquad Ric_{g,h,N} : =  Ric_{g} - (N-n) \frac{\nabla_{g}^{2} h^{\frac{1}{N-n}}}{h^{\frac{1}{N-n}}}.  
$$
In particular if $N = n$ the generalized Ricci tensor $Ric_{g,h,N}= Ric_{g}$ makes sense only if $h$ is constant. In particular, if $I \subset \R$ is any interval, $h \in C^{2}(I)$ 
and $\mathcal{L}^{1}$ is the one-dimensional Lebesgue measure, the m.m.s. $(I ,|\cdot|, h \mathcal{L}^{1})$ verifies $\CD(K,N)$ if and only if  
\begin{equation}\label{E:CD-N-1}
\left(h^{\frac{1}{N-1}}\right)'' + \frac{K}{N-1}h^{\frac{1}{N-1}} \leq 0.
\end{equation}
%

We also mention the more recent Riemannian curvature dimension condition $\RCD^{*}$ introduced in the infinite dimensional case in \cite{AGS, AGS11b,AGMR12} and then  investigated by various authors in the finite dimensional refinement. A remarkable property is the equivalence of the $\RCD^{*}(K,N)$ condition and the  Bochner inequality: the infinite dimensional case was settled in \cite{AGS11b}, while the (technically more involved) finite dimensional refinement was established    in \cite{EKS, AMS}. We refer to these papers and references therein for a general account 
on the synthetic formulation of Ricci curvature lower bounds for metric measure spaces. 

Here we only mention that $\RCD^{*}(K,N)$ condition 
is an enforcement of the so called reduced curvature dimension condition, denoted by $\CD^{*}(K,N)$, that has been introduced in \cite{BS10}: 
in particular the additional condition is that the Sobolev space $W^{1,2}(X,\mm)$ is an Hilbert space, see \cite{AGS11a, AGS11b}.

The reduced $\CD^{*}(K,N)$ condition asks for the same inequality \eqref{E:CD} of $\CD(K,N)$ but  the
coefficients $\tau_{K,N}^{(t)}(\sfd(\gamma_{0},\gamma_{1}))$ and $\tau_{K,N}^{(1-t)}(\sfd(\gamma_{0},\gamma_{1}))$ 
are replaced by $\sigma_{K,N}^{(t)}(\sfd(\gamma_{0},\gamma_{1}))$ and $\sigma_{K,N}^{(1-t)}(\sfd(\gamma_{0},\gamma_{1}))$, respectively.

Hence while the distortion coefficients of the $\CD(K,N)$ condition 
are formally obtained imposing one direction with linear distortion and $N-1$ directions affected by curvature, 
the $\CD^{*}(K,N)$ condition imposes the same volume distortion in all the $N$ directions.

For both definitions there is a local version that is of some relevance for our analysis. Here we state only the local formulation $\mathsf{CD}(K,N)$, 
being clear what would be the one for $\mathsf{CD}^{*}(K,N)$.

\begin{definition}[$\CD_{loc}$ condition]\label{D:loc}
An essentially non-branching m.m.s. $(X,\sfd,\mm)$ satisfies $\CD_{loc}(K,N)$ if for any point $x \in X$ there exists a neighborhood $X(x)$ of $x$ such that for each pair 
$\mu_{0}, \mu_{1} \in \mathcal{P}_{2}(X,\sfd,\mm)$ supported in $X(x)$
there exists $\nu \in \Opt(\mu_{0},\mu_{1})$ such that \eqref{E:CD} holds true for all $t \in [0,1]$.
The support of $({\rm e}_{t})_\sharp \, \nu$ is not necessarily contained in the neighborhood $X(x)$.
\end{definition}

One of the main properties of the reduced curvature dimension condition is the globalization one:  
under the non-branching property,  $\mathsf{CD}^{*}_{loc}(K,N)$ and $\mathsf{CD}^{*}(K,N)$ are equivalent (see \cite[Corollary 5.4]{BS10}), i.e. the 
$\mathsf{CD}^{*}$-condition verifies the local-to-global property.

We also recall a few relations between $\CD$ and $\CD^{*}$.
It is known by \cite[Theorem 2.7]{GigliMap} that, if $(X,\sfd,\mm)$ is a non-branching metric measure space 
verifying $\CD(K,N)$ and $\mu_{0}, \mu_{1} \in \mathcal{P}(X)$ with $\mu_{0}$ absolutely continuous with respect to $\mm$, 
then there exists a unique optimal map $T : X \to X$ such $(id, T)_\sharp\, \mu_{0}$ realizes the minimum in \eqref{eq:Wdef} and the set 
$\Opt(\mu_{0},\mu_{1})$ contains only one element. The same result holds if one replaces the non-branching assumption with the more general 
one of essentially non-branching (see \cite{CM3}); the same comment applies also to the previous equivalence between the local and the global version of $\CD^{*}(K,N)$.

\begin{remark}[$\CD^*(K,N)$ Vs $\CD_{loc}(K,N)$]\label{rk:CDCDs}
From \cite{BS10, CM3} we deduce the following chain of implications: if $(X,\sfd, \mm)$ is a proper, essentially non-branching, metric measure space, 
then 
$$
\CD_{loc}(K-,N) \iff \CD^{*}_{loc}(K-,N) \iff \CD^{*}(K,N), 
$$
provided $K,N \in \R$ with $N > 1$ or $N=1$ and $K \geq 0$. Let us remark  that on the other hand $\CD^*(K,1)$ does not imply $\CD_{loc}(K,1)$ for $K<0$: indeed it is possible to check that $(X,\sfd,\mm)=([0,1], |\cdot|, c\, \sinh(\cdot) \L^1)$ satisfies $\CD^*(-1,1)$ but not $\CD_{loc}(-1,1)$ which would require the density to be constant.
For a deeper analysis on the interplay between $\CD^{*}$ and $\CD$ we refer to \cite{BS10,GRS2013}.
\end{remark}

\subsection{Measured Gromov-Hausdorff convergence and stability of $\RCD^*(K,N)$}\label{SS:mGHConv}

Let us first recall the notion of measured Gromov-Hausdorff convergence, mGH for short. Since in this work we will apply it to compact  m.m. spaces endowed with  probability measures having full support, we will restrict to this framework for simplicity (for a more general treatment see for instance \cite{GMS2013}). We denote $\bar{\N}:=\N\cup\{\infty\}$.
   
 \begin{definition}  
 A sequence $(X_j,\sfd_j,\mm_j)$ of compact m.m. spaces with $\mm_j(X_j)=1$ and $\supp(\mm_j)=X_j$ is said to converge 
in the  measured Gromov-Hausdorff topology (mGH for short) to a compact m.m. space 
$(X_\infty,\sfd_\infty,\mm_\infty)$ with $\mm_\infty(X)=1$ and $\supp(\mm_\infty)=X_\infty$ if and only if there 
exists a separable metric space $(Z,\sfd_Z)$ and isometric embeddings  
$\{\iota_j:(X,\sfd_j)\to (Z,\sfd_Z)\}_{j \in \bar{\N}}$ with the following property:
for every 
$\varepsilon>0$  there exists $j_0$ such that for every $j>j_0$
\[
\iota_\infty(X_\infty) \subset B^Z_{\varepsilon}[\iota_j (X_j)]  \qquad \text{and} \qquad  \iota_j(X_j) \subset B^Z_{\varepsilon}[\iota_\infty(X_\infty)], 
\]
where $B^Z_\varepsilon[A]:=\{z \in Z: \, \sfd_Z(z,A)<\varepsilon\}$ for every subset $A \subset Z$, and 
\[
\int_Z \varphi \,  (\iota_j)_\sharp(\mm_j) \qquad    \to \qquad  \int_Z \varphi \,   (\iota_\infty)_\sharp(\mm_\infty) \qquad \forall \varphi \in C_b(Z), 
\]
where $C_b(Z)$ denotes the set of real valued bounded continuous functions with bounded support in $Z$.
 \end{definition}
 
The following theorem summarizes the compactness/stability properties we will use  in the proof of the almost rigidity result (notice these hold more generally for every $K\in \R$ by replacing mGH with  \emph{pointed}-mGH convergence).
\begin{theorem}[Metrizability and Compactness]\label{thm:CompRCD}
Let $K>0, N>1$ be fixed.  Then the mGH convergence restricted to  (isomorphism classes of)  $\RCD^*(K,N)$ spaces is metrizable by a distance function $\sfd_{mGH}$. Furthermore every sequence $(X_j,\sfd_j, \mm_j)$ of $\RCD^*(K,N)$ spaces admits a subsequence which  mGH-converges  to a limit $\RCD^*(K,N)$ space.
\end{theorem}

The compactness follows by the standard argument of Gromov, indeed for fixed $K>0,N>1$, the spaces have uniformly bounded diameter,  moreover  the measures of $\RCD^*(K,N)$  spaces  are uniformly doubling, hence the spaces  are uniformly totally bounded and thus compact in the GH-topology; the weak compactness of the measures follows using the doubling condition again and the fact that they are normalized. For the stability of the $\RCD^*(K,N)$ condition under mGH convergence see for instance   \cite{BS10,EKS,GMS2013}. The metrizability of mGH-convergence restricted to a class of  uniformly doubling  normalized m.m. spaces having uniform diameter bounds is also well known, see for instance \cite{GMS2013}.

\subsection{Warped product} 

Given two geodesic m.m.s. $(B,\sfd_{B}, \mm_{B})$ and $(F,\sfd_{F},\mm_{F})$ and a Lipschitz function $f : B \to \R_{+}$ one can define a 
length function on the product $B \times F$: for any absolutely continuous $\gamma : [0,1] \to B \times F$ with $\gamma = (\alpha, \beta)$, 
define 
$$
L(\gamma) : = \int_{0}^{1} \left(  |\dot \alpha|^{2}(t) + (f\circ \alpha)^{2}(t) |\dot \beta|^{2}(t) \right)^{1/2} dt
$$
and define accordingly the pseudo-distance 
$$
|(p,x),(q,y)| : = \inf \left\{ L(\gamma) \colon \gamma_{0} = (p,x), \ \gamma_{1} = (q,y) \right\}.
$$
Then the warped product of $B$ with $F$  is defined as 
$$
B \times_{f} F : = \left( \bigslant{B\times F}{ \sim}, |\cdot, \cdot | \right),
$$
where $(p,x) \sim (q,y)$ if and only if $|(p,x), (q,y)| = 0$. One can also associate a measure and obtain the following object
$$
B\times^{N}_{f} F : = (B \times_{f} F, \mm_{C}), \qquad \mm_{C} : = f^{N} \mm_{B} \otimes \mm_{F}. 
$$
Then  $B\times^{N}_{f} F$ will be a metric measure space called measured warped product. For a general picture on the curvature properties of warped products, 
we refer to \cite{Ket}.

\subsection{Isoperimetric profile}

Given a m.m.s. $(X,\sfd,\mm)$ as above and  a Borel subset $A\subset X$, let $A^{\ve}$ denote the $\ve$-tubular neighborhood 
$$
A^{\ve}:=\{x \in X \,:\, \exists y \in A \text{ such that } \sfd(x,y) < \ve \}. 
$$
The Minkowski (exterior) boundary measure $\mm^+(A)$  is defined by
\begin{equation}\label{eq:MinkCont}
\mm^+(A):=\liminf_{\ve\downarrow 0} \frac{\mm(A^\ve)-\mm(A)}{\ve}.
\end{equation}
The \emph{isoperimetric profile}, denoted by  ${\cI}_{(X,\sfd,\mm)}$, is defined as the point-wise maximal function so that $\mm^+(A)\geq \cI_{(X,\sfd,\mm)}(\mm(A))$  
for every Borel set $A \subset X$, that is
\begin{equation}\label{E:profile}
\cI_{(X,\sfd,\mm)}(v) : = \inf \big\{ \mm^{+}(A) \colon A \subset X \, \textrm{ Borel}, \, \mm(A) = v   \big\}.
\end{equation}

\subsection{The model Isoperimetric profile function $\cI_{K,N,D}$}\label{SS:IKND}
If $K>0$ and $N\in \N$, by the L\'evy-Gromov isoperimetric inequality \eqref{eq:LevyGromov} we know that, for $N$-dimensional smooth manifolds having Ricci $\geq K$, the isoperimetric profile function is bounded below by the one of the $N$-dimensional round sphere of the suitable radius. In other words  the \emph{model} isoperimetric profile function is the one of ${\mathbb S}^N$. For $N\geq 1, K\in \R$ arbitrary real numbers the situation is  more complicated, and just recently E. Milman \cite{Mil} discovered what is the model isoperimetric profile. In this short section we recall its definition.
\\

Given $\delta>0$, set 
\[
\begin{array}{ccc}
 s_\delta(t) := \begin{cases}
\sin(\sqrt{\delta} t)/\sqrt{\delta} & \delta > 0 \\
t & \delta = 0 \\
\sinh(\sqrt{-\delta} t)/\sqrt{-\delta} & \delta < 0
\end{cases}

& , &

 c_\delta(t) := \begin{cases}
\cos(\sqrt{\delta} t) & \delta > 0 \\
1 & \delta = 0 \\
\cosh(\sqrt{-\delta} t) & \delta < 0
\end{cases}
\end{array} ~.
\]
Given a continuous function $f :\R \to  \R$ with $f(0) \geq 0$, we denote by $f_+ : \R \to \R^+ $ the function coinciding with $f$ between its first non-positive and first positive roots, and vanishing everywhere else, i.e. $f_+ := f \chi_{[\xi_{-},\xi_{+}]}$ with $\xi_{-} = \sup\{\xi \leq 0; f(\xi) = 0\}$ and $\xi_{+} = \inf\{\xi > 0; f(\xi) = 0\}$.

Given $H,K \in \R$ and $N \in [1,\infty)$, set $\delta := K / (N-1)$  and define the following (Jacobian) function of $t \in \R$:
\[
J_{H,K,N}(t) :=
\begin{cases}
\chi_{\{t=0\}}  & N = 1 , K > 0 \\
\chi_{\{H t \geq 0}\} & N = 1 , K \leq 0 \\
\left(c_\delta(t) + \frac{H}{N-1} s_\delta(t)\right)_+^{N-1} & N \in (1,\infty) \\
\end{cases} ~.
\]
As last piece of notation, given a non-negative integrable function $f$ on a closed interval $L \subset \R$, we denote with $\mu_{f,L}$  
the probability measure supported in $L$ with density (with respect to the Lebesgue measure) proportional to $f$ there. In order to simplify a bit the notation we will write
$\cI_{(L,f)}$ in place of $\cI_{(L,\, |\cdot|, \mu_{f,L})}$.
\\The model isoperimetric profile for spaces having Ricci $\geq K$, for some $K\in \R$, dimension bounded above by $N\geq 1$ and diameter at most $D\in (0,\infty]$ is then defined by
\begin{equation}\label{eq:defIKND}
\cI_{K,N,D}(v):=\inf_{H\in \R,a\in [0,D]} \cI_{\left([-a,D-a], J_{H,K,N}\right)} (v), \quad \forall v \in [0,1]. 
\end{equation}
The formula above has the advantage of considering all the possible cases in just one equation, 
but  probably it is  also instructive to  isolate the different cases in a more explicit way. Indeed one can check \cite[Section 4] {Mil} that:
\begin{itemize}

\item \textbf{Case 1}: $K>0$ and $D<\sqrt{\frac{N-1}{K}} \pi$,
$$
\cI_{K,N,D}(v) =   \inf_{\xi \in \big[0, \sqrt{\frac{N-1}{K}} \pi -D\big]} \cI_{\big( [\xi,\xi +D],  \sin( \sqrt{\frac{K}{N-1}} t)^{N-1} \big)}(v), \quad \forall v \in [0,1] ~. 
$$

\item \textbf{Case 2}:   $K > 0$ and $D \geq \sqrt{\frac{N-1}{K}} \pi$, 
$$
\cI_{K,N,D}(v) = \cI_{ \big( [0, \sqrt{\frac{N-1}{K}} \pi],  \sin( \sqrt{\frac{K}{N-1}} t)^{N-1}   \big)}(v), \quad \forall v \in [0,1] ~. 
$$

\item  \textbf{Case 3}: $K=0$ and $D<\infty$,
\begin{eqnarray*}
\cI_{K,N,D}(v)&=&  \min \left \{ \begin{array}{l}  \inf_{\xi \geq 0} \cI_{([\xi,\xi+D],  t^{N-1})}(v) ~,\\
\phantom{\inf_{\xi \in \R}} \cI_{([0,D],1)}(v)
\end{array}
\right \} \\
&=& \frac{N}{D} \inf_{\xi \geq 0}  \frac{\left(\min(v,1-v) (\xi+1)^{N} + \max(v,1-v) \xi^{N}\right)^{\frac{N-1}{N}}}{(\xi+1)^{N} - \xi^{N}},  \quad \forall v \in [0,1] ~. 
\end{eqnarray*}

\item \textbf{Case 4}:  $K < 0$, $D<\infty$:
$$
\cI_{K,N,D}(v)=  \min \left \{ \begin{array}{l}
\inf_{\xi \geq 0} \cI_{\big([\xi,\xi+D], \; \sinh(\sqrt{\frac{-K}{N-1}} t)^{N-1}\big)}(v) ~ , \\
\phantom{\inf_{\xi \in \R}} \cI_{\big([0,D],   \exp(\sqrt{-K (N-1)} t) \big)} (v) ~, \\
\inf_{\xi \in \R} \cI_{\big([\xi,\xi+D], \; \cosh(\sqrt{\frac{-K}{N-1}} t)^{N-1}\big)} (v)
\end{array}
\right \} \quad \forall v \in [0,1]  ~.
$$
\item In all the remaining cases,  that is for $K\leq 0, D=\infty$,  the model profile trivializes: $\cI_{K,N.D}(v)=0$ for every $v \in [0,1].$
\end{itemize}
\medskip
Note that when $N$ is an integer, $$\cI_{\big( [0,  \sqrt{\frac{N-1}{K}} \pi ], ( \sin(\sqrt{\frac{K}{N-1} } t)^{N-1}\big)} = \cI_{({\mathbb S}^{N}, g^K_{can}, \mu^K_{can})}$$
 by the isoperimetric inequality on the sphere, and so Case 2 with $N$ integer corresponds to L\'evy-Gromov isoperimetric inequality.

\subsection{Disintegration of measures}
We include here a version of the Disintegration Theorem (for a comprehensive treatment see for instance  \cite{Fre:measuretheory4}).

Given a measurable space $(R, \mathscr{R})$, i.e.  $\mathscr{R}$ is  a $\sigma$-algebra of subsets of $R$, 
and a function $\QQ : R \to Q$, with $Q$ general set, we can endow $Q$ with the \emph{push forward $\sigma$-algebra} $\mathcal{Q}$ of $\mathscr{R}$:
$$
C \in \mathcal{Q} \quad \Longleftrightarrow \quad \QQ^{-1}(C) \in \mathscr{R},
$$
which could be also defined as the biggest $\sigma$-algebra on $Q$ such that $\QQ$ is measurable. 
Moreover given a probability measure  $\rho$  on $(R,\mathscr{R})$, define a probability  
measure $\qq$ on $(Q,\mathcal{Q})$  by push forward via $\QQ$, i.e. $\qq := \QQ_\sharp \, \rho$.  

\begin{definition}
\label{defi:dis}
A \emph{disintegration} of $\rho$ \emph{consistent with} $\QQ$ is a map 
(with slight abuse of notation still denoted by) $\rho: \mathscr{R} \times Q \to [0,1]$ such that, setting $\rho_q(B):=\rho(B,q)$, the following hold:
\begin{enumerate}
\item  $\rho_{q}(\cdot)$ is a probability measure on $(R,\mathscr{R})$ for all $q \in Q$,
\item  $\rho_{\cdot}(B)$ is $\qq$-measurable for all $B \in \mathscr{R}$,
\end{enumerate}
and satisfies for all $B \in \mathscr{R}, C \in \mathcal{Q}$ the consistency condition
$$
\rho \left(B \cap \QQ^{-1}(C) \right) = \int_{C} \rho_{q}(B)\, \qq(dq).
$$
A disintegration is \emph{strongly consistent with respect to $\QQ$} if for all $q$ we have $\rho_{q}(\QQ^{-1}(q))=1$.
The measures $\rho_q$ are called \emph{conditional probabilities}.
\end{definition}

We recall the following version of the disintegration theorem that can be found in \cite[Section 452]{Fre:measuretheory4}.
Recall that a $\sigma$-algebra $\mathcal{J}$ is \emph{countably generated} if there exists a countable family of sets so that 
$\mathcal{J}$ coincide with the smallest $\sigma$-algebra containing them.

\begin{theorem}[Disintegration of measures]
\label{T:disintr}
Assume that $(R,\mathscr{R},\rho)$ is a countably generated probability space and  $R = \cup_{q \in Q}R_{q}$ is a partition of R. 
Denote with $\QQ: R \to Q$ the quotient map: 
$$
 q = \QQ(x) \iff x \in R_{q},
$$
and with $\left( Q, \mathcal{Q},\qq \right)$ the quotient measure space. 
Assume $(Q,\mathcal{Q})=(X,\mathcal{B}(X))$ with $X$ Polish space, where $\mathcal{B}(X)$ denotes the Borel $\sigma$-algebra.
Then there exists a unique strongly consistent disintegration $q \mapsto \rho_{q}$ w.r.t. $\QQ$, where uniqueness is understood in the following sense:
if $\rho_{1}, \rho_{2}$ are two consistent disintegrations then $\rho_{1,q}(\cdot)=\rho_{2,q}(\cdot)$ for $\qq$-a.e. $q \in Q$.
\end{theorem}


\section{$\sfd$-monotone sets}\label{S:dmonotone}

Let $\f : X \to \R$ be any $1$-Lipschitz function. Here  we present some useful results concerning the $\sfd$-cyclically monotone set associated with $\f$:
\begin{equation}\label{E:Gamma}
\Gamma : = \{ (x,y) \in X\times X : \f(x) - \f(y) = \sfd(x,y) \},
\end{equation}
that can be interpret as the set of couples moved by $\f$ with maximal slope.
Recall that a set $\Lambda \subset X \times X$ is said to be $\sfd$-cyclically monotone if for any finite set of points $(x_{1},y_{1}),\dots,(x_{N},y_{N})$ it holds
$$
\sum_{i = 1}^{N} \sfd(x_{i},y_{i}) \leq \sum_{i = 1}^{N} \sfd(x_{i},y_{i+1}),
$$
with the convention that $y_{N+1} = y_{1}$. \medskip

The following lemma is a consequence of the $\sfd$-cyclically monotone structure of $\Gamma$.


\begin{lemma}\label{L:cicli}
Let $(x,y) \in X\times X$ be an element of $\Gamma$. Let $\gamma \in \Geo(X)$ be such that $\gamma_{0} = x$ 
and $\gamma_{1}=y$. Then
$$
(\gamma_{s},\gamma_{t}) \in \Gamma, 
$$
for all $0\leq s \leq t \leq 1$.
\end{lemma}

For its proof see Lemma 3.1 of \cite{cava:MongeRCD}. It is natural then to consider the set of geodesics $G \subset \Geo(X)$ such that 
$$
\gamma \in G \iff \{ (\gamma_{s},\gamma_{t}) : 0\leq s \leq t \leq 1  \} \subset \Gamma,
$$
that is $G : = \{ \gamma \in \Geo(X) : (\gamma_{0},\gamma_{1}) \in \Gamma \}$.  \\
We now recall some definitions, already given in \cite{biacava:streconv},  that will be needed to describe the structure of $\Gamma$.

\begin{definition}\label{D:transport}
We define the set of \emph{transport rays} by 
$$
R := \Gamma \cup \Gamma^{-1},
$$
where $\Gamma^{-1}:= \{ (x,y) \in X \times X : (y,x) \in \Gamma \}$. 
The set of \emph{initial points} and \emph{final points} are defined   respectively by  
\begin{align*}
{\mathfrak a} :=& \{ z \in X: \nexists \, x \in X, (x,z) \in \Gamma, \sfd(x,z) > 0  \}, \crcr
{\mathfrak b} :=& \{ z \in X: \nexists \, x \in X, (z,x) \in \Gamma, \sfd(x,z) > 0 \}.
\end{align*}
The set of \emph{end points} is ${\mathfrak a} \cup {\mathfrak b}$. 
We define the \emph{transport set with end points}: 
$$
\mathcal{T}_{e} = P_{1}(\Gamma \setminus \{ x = y \}) \cup 
P_{1}(\Gamma^{-1}\setminus \{ x=y \}).
$$
where $\{ x = y\}$ stands for $\{ (x,y) \in X^{2} : \sfd(x,y) = 0 \}$.
\end{definition}

\begin{remark}\label{R:regularity}
Here we discuss the measurability of the sets introduced in Definition \ref{D:transport}.
Since $\f$ is $1$-Lipschitz, $\Gamma$ is closed and therefore $\Gamma^{-1}$ and $R$ are closed as well.
Moreover by assumption the space is proper, hence the sets $\Gamma, \Gamma^{-1}, R$ are $\sigma$-compact (countable union of compact sets).

Then we look at the set of initial and final points:
$$
{\mathfrak a} = P_{2} \left( \Gamma \cap \{ (x,z) \in X\times X : \sfd(x,z) > 0 \} \right)^{c}, \qquad 
{\mathfrak b} = P_{1} \left( \Gamma \cap \{ (x,z) \in X\times X : \sfd(x,z) > 0 \} \right)^{c}.
$$
Since $\{ (x,z) \in X\times X : \sfd(x,z) > 0 \} = \cup_{n} \{ (x,z) \in X\times X : \sfd(x,z) \geq 1/n \}$, it follows that  
it follows that both ${\mathfrak a}$ and ${\mathfrak b}$ are the complement of $\sigma$-compact sets. 
Hence ${\mathfrak a}$ and ${\mathfrak b}$ are Borel sets. Reasoning as before, it follows that $\T_{e}$ is a $\sigma$-compact set.
\end{remark}

It can be proved that the set of transport rays $R$ induces an equivalence relation on a subset of $\mathcal{T}_{e}$. It is sufficient to remove from $\mathcal{T}_{e}$
the branching points of geodesics and then show that they all have $\mm$-measure zero. This will be indeed the case using the curvature 
properties of the space.

To this aim, set $\Gamma(x):=P_2(\Gamma\cap(\{x\}\times X))$,  $\Gamma(x)^{-1}:=P_2(\Gamma^{-1}\cap(\{x\} \times X))$, consider the sets of forward (respectively backward) branching 
\begin{align*}
	A_{+}: = 	&~\{ x \in \mathcal{T}_{e} : \exists z,w \in \Gamma(x), (z,w) \notin R \}, \\
	A_{-}	: = 	&~\{ x \in \mathcal{T}_{e} : \exists z,w \in \Gamma(x)^{-1}, (z,w) \notin R \},
\end{align*}
and define the \emph{transport set} $\mathcal{T} : = \mathcal{T}_{e} \setminus (A_{+} \cup A_{-})$. Then one can prove the following

\begin{theorem}\label{T:equivalence}
Let $(X,\sfd,\mm)$ satisfy $\CD^{*}(K,N)$ and be essentially non-branching with $1 \leq N < \infty$.
Then the set of transport rays $R\subset X \times X$ is an equivalence relation on the transport set $\mathcal{T}$ 
and 
$$
\mm(\mathcal{T}_{e} \setminus \mathcal{T} ) = 0.
$$
Moreover  the transport set $\mathcal{T}$ is $\sigma$-compact set. 
\end{theorem}

For its proof in the context of $\RCD$-space see \cite[Theorem 5.5]{cava:MongeRCD}; the proof works here (see \cite{CM3}).


\noindent
The next step is to decompose the reference measure $\mm$ restricted to $\mathcal{T}$ 
with respect to the partition given by $R$, where each equivalence class is given by
$$
[x]  =  \{ y \in \mathcal{T}: (x,y) \in R \}.
$$
%
%
%
Denote the set of equivalence classes with $Q$. In order to use Disintegration Theorem, we need to construct the quotient map 
$$
\QQ : \mathcal{T} \to Q
$$
associated to the equivalence relation $R$. 
%
Recall that a \emph{section of an equivalence relation $E$ over $\mathcal{T}$} is a map $F : \mathcal{T} \to \mathcal{T}$ such that for any 
$x,y \in X$ it holds
$$
(x,F(x)) \in E, \qquad (x,y)\in E \Rightarrow F(x) =F(y).
$$
Note that to each section $F$ is canonically associated a quotient set $Q = \{ x \in \mathcal{T} : x=F(x) \}.$

\begin{proposition}\label{P:section} 
There exists an $\mm$-measurable section 
$$
\QQ : \mathcal{T} \to \mathcal{T}
$$
for the equivalence relation $R$. 
\end{proposition}
For its proof see \cite[Proposition 5.2]{cava:MongeRCD}.

\noindent
As pointed out before, one can take as quotient space $Q$ the image of $\QQ$ and since 
$$
Q = \QQ(\mathcal{T} ) = \{ x \in \mathcal{T}  : \sfd(x,\QQ(x)) = 0 \},
$$
it follows that $Q$ is $\mm$-measurable. Then the quotient measure will be given by
$$
\qq : = \QQ_\sharp  \, \mm\llcorner_{\mathcal{T}} .
$$
Observe that from the $\mm$-measurability of $\QQ$ it follows that $\qq$ is a Borel measure.
By inner regularity of compact sets, one can find a $\sigma$-compact set $S \subset Q$ such that $\qq(Q \setminus S)=0$. 
By definition of $\qq$ it follows that $\mm (\mathcal{T} \setminus \QQ^{-1}(S)) = 0$, 
in particular one can take a Borel subset of the quotient set without changing $\mm\llcorner_{\mathcal{T}}$. 

Then from Theorem \ref{T:disintr} one obtains the  following disintegration formula,
\begin{equation}\label{E:disint}
\mm \llcorner_{\mathcal{T}} = 
			\int_{Q} \mm_{q} \, \qq(dq), \qquad \mm_{q}(\QQ^{-1}(q)) = 1, \ \qq\text{-a.e.}\ q \in Q.
\end{equation}
We now consider the ray map from \cite{biacava:streconv}, Section 4.
\begin{definition}[Ray map]
\label{D:mongemap}
Define the \emph{ray map}  
$$
g:  \textrm{Dom}(g) \subset S \times \R \to \mathcal{T}
$$ 
via the formula:
\begin{align*}
\gr (g) : 	= 	&~ \Big\{ (q,t,x) \in S \times [0,+\infty) \times \mathcal{T}:  (q,x) \in \Gamma, \, \sfd(q,x) = t  \Big\} \crcr
			&~ \cup \Big\{ (q,t,x) \in S \times (-\infty,0] \times \mathcal{T}  : (x,q) \in \Gamma, \, \sfd(x,q) = t \Big\} \crcr
		=	&~ \gr(g^+) \cup \gr(g^-).
\end{align*}
\end{definition}
Hence the ray map associates to each $q \in S$ and $t\in  \dom(g(q, \cdot))\subset \R$ the unique element $x \in \mathcal{T}$ such that $(q,x) \in \Gamma$ at distance $t$ from $q$
if $t$ is positive or the unique element $x \in \mathcal{T}$ such that $(x,q) \in \Gamma$ at distance $-t$ from $q$
if $t$ is negative. By definition $\textrm{Dom}(g) : = g^{-1}(\mathcal{T})$.

Next we list few regularity properties enjoyed by $g$ (\cite[Proposition 5.4]{cava:MongeRCD}).
\begin{proposition} \label{P:gammaclass}
The following holds.
\begin{itemize}
\item[-] $g$ is a Borel map.
\item[-] $t \mapsto g(q,t)$ is an isometry and if $s,t \in \dom(g(q,\cdot))$ with $s \leq t$ then $( g(q,s), g(q,t) ) \in \Gamma$;
\item[-] $\textrm{Dom}(g) \ni (q,t) \mapsto g(q,t)$ is bijective on $\QQ^{-1}(S) \subset \mathcal{T}$, and its inverse is
$$
x \mapsto g^{-1}(x) = \big( \QQ(x),\pm \sfd(x,\QQ(x)) \big)
$$
where $\QQ$ is the quotient map previously introduced and the positive or negative sign depends on 
$(x,\QQ(x)) \in \Gamma$ or $(\QQ(x),x) \in \Gamma$.
\end{itemize}
\end{proposition}
Observe that from Lemma \ref{L:cicli},  $\dom (g(q,\cdot))$ is a convex subset of $\R$ (i.e. an interval), 
for any $q \in Q$. Using the ray map $g$ one can prove that $\qq$-almost every conditional measure $\mm_{q}$ is absolutely continuous with respect to the $1$-dimensional 
Hausdorff measure  considered on the ray passing through $q$. This and all the other results presented so far are contained in the next theorem.

\begin{theorem}\label{T:summary}
Let $(X,\sfd,\mm)$ verify $\CD^{*}(K,N)$ for some $K,N \in \R$, with $1 \leq N < \infty$ and be essentially non-branching.
Let moreover $\Gamma$ be a $\sfd$-cyclically monotone set such as in \eqref{E:Gamma} and let $\mathcal{T}_{e}$ be the set of all points moved by $\Gamma$ as in Definition \ref{D:transport}.
\medskip

Then there exists $\mathcal{T} \subset \mathcal{T}_{e}$ that we call \emph{transport set} such that 
\begin{enumerate}
\item $\mm(\mathcal{T}_{e} \setminus \mathcal{T}) = 0,$ 
\item  for every $x \in \mathcal{T}$, the transport ray $R(x):=\QQ^{-1}(\QQ(x))$ is formed by a single geodesic and for $x\neq y$, both in $\mathcal{T}$, either $R(x) = R(y)$
or $R(x) \cap R(y)$ is contained in the set of forward and backward branching points ${A_{+}} \cup {A_{-}}$.
\end{enumerate}
Moreover the following disintegration formula holds
$$
\mm \llcorner_{\mathcal{T}} = 	\int_{Q} \mm_{q} \, \qq(dq), \qquad \mm_{q}(\QQ^{-1}(q)) = 1, \ \qq\text{-a.e.}\ q \in Q.
$$
Finally for $\qq$-a.e. $q \in Q$ the conditional measure $\mm_{q}$ is absolutely continuous with respect to $\haus^{1}\llcorner_{\{ g(q,t) : t\in \R \}}$.
\end{theorem}

For its proof see \cite[Theorem 6.6.]{cava:MongeRCD}. 
Note that in Theorem 1.1 and Theorem 5.5 of \cite{cava:MongeRCD} it is uncorrectly stated that $R(x) \cap R(y)$ 
is contained in the set of end points ${\mathfrak a} \cup {\mathfrak b}$, see Definition \ref{D:transport}, while the proof yields the weaker version that we just reported. \\
Notice that since $t \mapsto g(q,t)$ is an isometry, $\haus^{1}\llcorner_{\{ g(q,t) : t\in \R \}} = g(q,\cdot)_\sharp \, \mathcal{L}^{1}$.

\bigskip

We conclude this section showing that, locally, the quotient set $Q$ can be identified with a subset of a level set  of $\f$. 

\begin{lemma}[$Q$ is locally contained in level sets of $\f$]\label{lem:Qlevelset}
It is possible to construct a Borel quotient map $\QQ: \T \to Q$ such that the quotient set $Q$ can be written locally  as a level set of $\f$ in the following sense: 
$$
Q = \bigcup_{i\in \N} Q_{i}, \qquad Q_{i} \subset \f^{-1}(\alpha_{i}), 
$$
where $\alpha_i \in \Q$, $Q_{i}$ is  $\sigma$-compact (in particular  $Q_{i}$ is Borel) and $Q_{i} \cap Q_{j} = \emptyset$, for $i\neq j$.
\end{lemma}

\begin{proof}
For each $n \in \N$, consider the set $\T_{n}$ of those points $x$ having ray $R(x)$ longer than $1/n$, i.e.
$$
\T_{n} : = P_{1} \{ (x,y) \in \T \times \T \cap R \colon \sfd(x,y) \geq 1/n \}.
$$
It is easily seen that  $\T=\bigcup_{n \in \N} \T_n$ and that  $\T_{n}$ is $\sigma$-compact; moreover if $x \in \T_{n}, y \in \T$ and $(x,y) \in R$ then also $y \in \T_{n}$. 
In particular, $\T_{n}$ is the union of all those maximal rays of $\T$ with length at least $1/n$. 
Now we consider the following saturated subsets of $\T_{n}$:  for $\alpha \in \Q$
\begin{equation}\label{eq:defTnalpha}
 \T_{n,\alpha}:=  P_{1}  \Big( R \cap \Big \{ (x,y) \in \T_{n} \times \T_{n} \colon  \f(y) = \alpha - \frac{1}{3n}\Big \}  \Big)  \cap P_{1} \Big( R \cap \Big \{ (x,y) \in \T_{n} \times \T_{n} \colon  \f(y) = \alpha+  \frac{1}{3n} \Big\}  \Big),
\end{equation}
and we claim  that 
\begin{equation}  \label{eq:Tnalpha}
\T_{n} =   \bigcup_{\alpha \in \Q}  \T_{n,\alpha} . 
\end{equation}

We show the above identity by double inclusion. First note that $(\supset)$ holds trivially. For the converse inclusion $(\subset)$   observe that
for each $\alpha \in \Q$, the set $ \T_{n,\alpha}$ coincides with the family of those rays $R(x) \cap \T_{n}$ such that there exists $y^{+},y^{-} \in R(x)$ such that
\begin{equation}\label{eq:ypm}
\f(y^{+}) = \alpha - \frac{1}{3n}, \qquad \f(y^{-}) = \alpha + \frac{1}{3n}. 
\end{equation}
Then we need to show that any $x \in \T_{n}$, also verifies $x \in  \T_{n,\alpha}$ for a suitable $\alpha \in \Q$. So fix $x \in \T_{n}$ and since $R(x)$ is longer than $1/n$, there exist $y,z^{+},z^{-} \in R(x) \cap \T_{n}$ such that
$$
\f(y) -\f(z^{+}) = \frac{1}{2n}, \qquad  \f(z^{-}) -\f(y)= \frac{1}{2n}. 
$$
In particular, if $y = g(q,s)$, the map $\big[-\frac{1}{2n}, \frac{1}{2n}\big] \ni t \mapsto \f(g(q, s+ t )) = \f(y) - t $ 
is well defined with image $\big[\f(y) - \frac{1}{2n}, \f(y) + \frac{1}{2n}\big]$.
If we pick $\alpha \in \Q$ such that $|\f(y) - \alpha| < \frac{1}{6n}$, it is straightforward to verify the existence of $t^{-},t^{+}\in \big[0, \frac{1}{2n}\big]$ such that 
$$
 \f(g(q, s+ t^{+} )) = \alpha - \frac{1}{3n}, \qquad \f(g(q, s- t^{-} ))  = \alpha + \frac{1}{3n}. 
$$
In particular, calling  $y^{+} = g(q, s+ t^{+} )$ and $y^{-} = g(q, s- t^{-} )$, we have that $y^{+}, y^{-}\in R(x)$ satisfy \eqref{eq:ypm} and therefore $x\in  \T_{n,\alpha}$.  This concludes the proof of the identity   \eqref{eq:Tnalpha}.  
\\By the above construction, one can check that for each $\alpha \in \Q$, the  level set $\f^{-1}(\alpha)$ is a quotient set for  $\T_{n,\alpha}$, i.e.
%
%
 $\T_{n,\alpha}$ is formed by disjoint geodesics  each one  intersecting $\f^{-1}(\alpha)$ in exactly one point. 
Observe moreover that $\T_{n,\alpha}$ is $\sigma$-compact. \\
Since $\T=\bigcup_{n \in \N} \T_n = \bigcup_{n \in \N, \alpha\in \Q} \T_{n,\alpha}$, it follows that we have just constructed a quotient map $\QQ:\T \to Q$ such that  the quotient set $Q$ satisfies the desired properties.  Moreover its graph verifies: 
$$
\gr(\QQ) = \bigcup_{n \in \N, \alpha\in \Q} \T_{n,\alpha} \times \left( \f^{-1}(\alpha) \cap\T_{n,\alpha} \right),
$$
in particular it is $\sigma$-compact and therefore $\QQ$ is Borel-measurable.
\end{proof}


\medskip
\section{$\sfd$-monotone sets and curved conditional measures} \label{sec:ConditionalMeasures}
In this Section we focus on the curvature properties of $\mm_{q}$.
Recall that to any $1$-Lipschitz function we associate a $\sfd$-monotone set to which in turn we associate a partition and a disintegration on the corresponding transport set: 
$$
\mm \llcorner_{\mathcal{T}} = 	\int_{Q} \mm_{q} \, \qq(dq), \qquad \mm_{q}(\QQ^{-1}(q)) = 1, \ \qq\text{-a.e.}\ q \in Q,
$$
and for $\qq$-a.e. $q \in Q$ 
$$
\mm_{q} = g(q,\cdot)_\sharp  \left(  h_{q} \cdot \mathcal{L}^{1} \right),  
$$
for some function $h_{q} :  \dom (g(q,\cdot)) \subset \R \to [0,\infty)$.  

It has already been shown that $h_{q}$ has some regularity properties, provided that the space verifies some curvature bounds. 
We start by recalling the following inequality obtained in \cite{cava:MongeRCD} in the context of $\RCD$-space; the same proof works here (see \cite{CM3}):  
assume $K >0$, for $\qq$-a.e. $q \in Q$:
\begin{equation}\label{E:MCP}
\left( \frac{\sin( (b - t_{1}) \sqrt{K/(N-1)}}{\sin( (b - t_{0}) \sqrt{K/(N-1)}} \right)^{N-1} 
\leq \frac{h_{q}(t_{1})}{h_{q}(t_{0})}
\leq \left( \frac{\sin( ( t_{1} - a) \sqrt{K/(N-1)}}{\sin( ( t_{0} - a) \sqrt{K/(N-1)}} \right)^{N-1} 
\end{equation}
for each $a < t_{0} < t_{1} < b$ and $a,b \in \dom(g(q,\cdot))$. A similar inequality also holds for $K \leq 0$.
It follows that for $\qq$-a.e.$q \in Q$,  
\begin{equation}\label{E:regularityh}
\{ t \in \dom(g(q,\cdot))  \colon  h_{q}(t) > 0 \} = \dom(g(q,\cdot)) \textrm{ is convex,}
\end{equation}
and $t \mapsto h_{q}(t)$ is locally Lipschitz continuous.


In order to deduce stronger curvature properties for the density $h_{q}$, one should use the full curvature information of the space. 
In order to do so it is necessary to include $\sfd^{2}$-cyclically monotone sets as subset of $\sfd$-cyclically monotone sets.
We present here a strategy already introduced in \cite{cava:MongeRCD}.

\begin{lemma}\label{L:12monotone}
Let $\Delta \subset \Gamma$ be any set so that: 
$$
(x_{0},y_{0}), (x_{1},y_{1}) \in \Delta \quad \Rightarrow \quad (\f(y_{1}) - \f(y_{0}) )\cdot (\f(x_{1}) - \f(x_{0}) ) \geq 0.
$$
Then $\Delta$ is $\sfd^{2}$-cyclically monotone.
\end{lemma}

\begin{proof} 
It follows directly from the hypothesis of the lemma that the set
$$
\Lambda: = \{ (\f(x), \f(y) ) :   (x,y) \in \Delta \} \subset \erre^{2}, 
$$
is monotone in the Euclidean sense. Since $\Lambda \subset \R^{2}$, it is then a standard fact that $\Lambda$ is 
also $|\cdot|^{2}$-cyclically monotone, where $|\cdot|$ denotes the modulus. 
We anyway  include a short proof:
there exists a maximal monotone multivalued function $F$ such that $\Lambda \subset \gr (F)$ and its domain is an interval, say $(a,b)$ with $a$ and $b$ possibly infinite; 
moreover, apart from countably many $x \in \R$, the set $F(x)$ is a singleton. 
Then the following function is well defined:  
$$
\Psi(x) : = \int_{c}^{x} F(s) ds, 
$$
where $c$ is any fixed element of $(a,b)$. Then observe that %
$$
\Psi(z) - \Psi(x) \geq y(z-x), \qquad \forall \ z,x \in (a,b),
$$
where $y$ is any element of $F(x)$. In particular this implies that $\Psi$ is convex and $F(x)$ is a subset of its sub-differential. 
In particular $\Lambda$ is $|\cdot |^{2}$-cyclically monotone. \\
Then for $\{(x_{i},y_{i})\}_{ i \leq N} \subset \Delta$, since $\Delta \subset \Gamma$,
it holds
\begin{align*} 
\sum_{i=1}^{N} \sfd^{2}(x_{i},y_{i}) = &~ \sum_{i =1}^{N}|\f(x_{i}) - \f(y_{i})|^{2} \crcr
\leq&~ \sum_{i =1}^{N}|\f(x_{i}) - \f(y_{i+1})|^{2} \crcr
\leq &~ \sum_{i=1}^{N} \sfd^{2}(x_{i},y_{i+1}),
\end{align*}
where the last inequality is given by the 1-Lipschitz regularity of $\f$. The claim follows.
\end{proof}

Before stating the next result let us recall that  $\CD^{*}(K,N)$ and $\CD_{loc}(K,N)$ are equivalent if  $1 < N <\infty$ or $N =1$ and $K \geq 0$, but for $N =1$ and $K < 0$ the $\CD_{loc}(K,N)$ condition is strictly stronger than $\CD^{*}(K,N)$, see Remark \ref{rk:CDCDs} for more details.

\begin{theorem}\label{T:CDKN-1}
Let $(X,\sfd,\mm)$ be an essentially non-branching m.m.s. verifying the $\CD_{loc}(K,N)$ condition for some $K\in \R$ and $N\in [1,\infty)$.

Then for any 1-Lipschitz function  $\f:X\to \R$, the associated transport set $\Gamma$ induces a disintegration 
of $\mm$ restricted to the transport set verifying the following: if $N> 1$,  for $\qq$-a.e. $q \in Q$ the following curvature inequality holds 
\begin{equation}\label{E:curvdensmm}
h_{q}( (1-s)  t_{0}  + s t_{1} )^{1/(N-1)}  
 \geq \sigma^{(1-s)}_{K,N-1}(t_{1} - t_{0}) h_{q} (t_{0})^{1/(N-1)} + \sigma^{(s)}_{K,N-1}(t_{1} - t_{0}) h_{q} (t_{1})^{1/(N-1)},
\end{equation}
for all $s\in [0,1]$ and for all $t_{0}, t_{1} \in \dom(g(q,\cdot))$ with  $t_{0} < t_{1}$. If $N =1$, for $\qq$-a.e. $q \in Q$ the density $h_{q}$ is constant.
\end{theorem}

\begin{proof}
We first consider the case $N>1$.
As $\CD^{*}(K,N)$ and $\CD_{loc}(K,N)$ are equivalent conditions in the framework of essentially non-branching metric measure spaces, 
during this proof we will use the convexity property imposed by $\CD_{loc}(K,N)$.

{\bf Step 1.} \\
Thanks to Lemma \ref{lem:Qlevelset}, without any loss of generality we can assume that the quotient set $Q$ (identified with the set $\{g(q,0) : q \in Q\}$) 
is locally a subset of a level set of the map $\f$ inducing the transport set, i.e.
there exists a countable partition $\{ Q_{i}\}_{i\in \N}$ with $Q_{i} \subset Q$ Borel set such that
$$
\{ g(q,0) : q \in Q_{i} \} \subset \{ x \in X : \f(x) = \alpha_{i} \}.
$$
It is clearly sufficient to prove  \eqref{E:curvdensmm} on each $Q_{i}$;  so  fix $\bar i \in \N$ and for ease of notation assume $\alpha_{\bar i} = 0$ and $Q = Q_{\bar i}$.
As $\dom(g(q,\cdot))$ is a convex subset of $\R$, we can also restrict to a uniform subinterval 
$$
 (a_0,a_1) \subset \dom(g(q,\cdot)), \qquad \forall \ q \ \in Q_{i},
$$
for some $a_0,a_1 \in \R$. Again without any loss of generality we also assume $a_0 < 0 < a_1$. 

\bigskip

Consider any $a_{0} <A_{0} < A_{1} < a_{1}$ and $L_{0}, L_{1} >0$ such that $A_{0} + L_{0} < A_{1}$ and $A_{1} + L_{1} < a_{1}$.
Then define the following two probability measures
$$
\mu_{0} : = \int_{Q} g(q,\cdot)_\sharp \left(  \frac{1}{L_{0}} \mathcal{L}^{1}\llcorner_{ [A_{0},A_{0}+L_{0}] } \right) \, \qq(dq), \qquad 
\mu_{1} : = \int_{Q} g(q,\cdot)_\sharp \left( \frac{1}{L_{1}} \mathcal{L}^{1}\llcorner_{ [A_{1},A_{1}+L_{1}] } \right) \, \qq(dq).
$$
Since $g(q,\cdot)$ is an isometry one can also represent $\mu_{0}$ and $\mu_{1}$ in the following way: 
$$
\mu_{i} : = \int_{Q} \frac{1}{L_{i}}  \mathcal{H}^{1}\llcorner_{ \left\{g(q,t) \colon t \in [A_{i},A_{i}+L_{i}] \right\} } \, \qq(dq)
$$
for $i =0,1$. Theorem \ref{T:summary} implies that $\mu_{i}$ is absolutely continuous with respect to $\mm$ and $\mu_{i} = \r_{i} \mm$
with 
$$
\r_{i} (g(q,t)) = \frac{1}{L_{i}} h_{q}(t)^{-1}, \qquad \forall \, t \in [A_{i},A_{i}+L_{i}]. 
$$
Moreover from Lemma \ref{L:12monotone} it follows that the curve $[0,1] \ni s \mapsto \mu_{s} \in \mathcal{P}(X)$  defined by 
$$
\mu_{s} : = \int_{Q} \frac{1}{L_{s}}  \mathcal{H}^{1}\llcorner_{ \left\{g(q,t) \colon t \in [A_{s},A_{s}+L_{s}] \right\} } \, \qq(dq)
$$
where 
$$
L_{s} : = (1 - s)L_{0} + sL_{1}, \qquad A_{s} : = (1-s ) A_{0} + s A_{1}
$$
is the unique $L^{2}$-Wasserstein geodesic connecting $\mu_{0}$ to $\mu_{1}$. Again one has $\mu_{s} = \r_{s} \mm$ and can also write its density in the following way:
$$
\r_{s} (g(q,t)) = \frac{1}{L_{s}} h_{q}(t)^{-1}, \qquad \forall \, t \in [A_{s},A_{s}+L_{s}]. 
$$

{\bf Step 2.}\\
By $\CD_{loc}(K,N)$ and the essentially non-branching property one has: for $\qq$-a.e. $q \in Q_{i}$
$$
(L_{s})^{\frac{1}{N}} h_{q}( (1-s) t_{0} + s t_{1} )^{\frac{1}{N}}
	\geq 		\tau_{K,N}^{(1-s)}(t_{1}-t_{0}) (L_{0})^{\frac{1}{N}} h_{q}( t_{0} )^{\frac{1}{N}}+  \tau_{K,N}^{(s)}(t_{1}-t_{0}) (L_{1})^{\frac{1}{N}} h_{q}( t_{1} )^{\frac{1}{N}}, 
$$
for $\L^{1}$-a.e. $t_{0} \in [A_{0},A_{0} + L_{0}]$ and $t_{1}$ obtained as the image of $t_{0}$ through the monotone rearrangement of $[A_{0},A_{0}+L_{0}]$ to 
$[A_{1},A_{1}+L_{1}]$ and every $s \in [0,1]$. If $t_{0} = A_{0} + \tau L_{0}$, then $t_{1} = A_{1} + \tau L_{1}$. Also $A_{0}$ and $A_{1} +L_{1}$ should be taken close enough to 
verify the local curvature condition. 

Then we can consider the previous inequality only for $s = 1/2$ and include the explicit formula for $t_{1}$ and obtain: 
\begin{align*}
(L_{0} + L_{1})^{\frac{1}{N}} &h_{q}(A_{1/2} + \tau L_{1/2})^{\frac{1}{N}} \\
			&	\geq  
		\sigma^{(1/2)}_{K,N-1}( A_{1} - A_{0} + \tau |L_{1} - L_{0}| )^{\frac{N-1}{N}} \left\{ (L_{0})^{\frac{1}{N}} h_{q}(A_{0} + \tau L_{0})^{\frac{1}{N}} 
			+ (L_{1})^{\frac{1}{N}} h_{q}(A_{1} + \tau L_{1})^{\frac{1}{N}} \right\},
\end{align*}
for $\L^{1}$-a.e. $\tau \in [0,1]$, where we used the notation $A_{1/2}:=\frac{A_0+A_1}{2}, L_{1/2}:=\frac{L_0+L_1}{2}$. Now observing that the map $s \mapsto h_{q}(s)$ is continuous (see \eqref{E:MCP}), the previous inequality also holds for $\tau =0$:
\begin{equation}\label{E:beforeoptimize}
(L_{0} + L_{1})^{\frac{1}{N}} h_{q}(A_{1/2} )^{\frac{1}{N}}
		\geq 
		\sigma^{(1/2)}_{K,N-1}( A_{1} - A_{0})^{\frac{N-1}{N}} 
				\left\{ (L_{0})^{\frac{1}{N}} h_{q}(A_{0})^{\frac{1}{N}} + (L_{1})^{\frac{1}{N}} h_{q}(A_{1})^{\frac{1}{N}} \right\},
\end{equation}
for all $A_{0} < A_{1}$  with $A_{0},A_{1}\in (a_0, a_1)$, all sufficiently small $L_{0}, L_{1}$ and $\qq$-a.e. $q\in Q$, 
with exceptional set depending on $A_{0},A_{1},L_{0}$ and $L_{1}$. 

Noticing that \eqref{E:beforeoptimize} depends in a continuous way on $A_{0},A_{1},L_{0}$ and $L_{1}$, it follows that there 
exists a common exceptional set $N \subset Q$ such that $\qq(N) = 0$ and for each $q \in Q\setminus N$ for all  
$A_{0},A_{1},L_{0}$ and $L_{1}$ the inequality \eqref{E:beforeoptimize} holds true.
Then one can make the following (optimal) choice 
$$
L_{0} : = L \frac{h_{q}(A_{0})^{\frac{1}{N-1}}  }{h_{q}(A_{0})^{\frac{1}{N-1}} + h_{q}(A_{1})^{\frac{1}{N-1}} }, \qquad 
L_{1} : = L \frac{h_{q}(A_{1})^{\frac{1}{N-1}}  }{h_{q}(A_{0})^{\frac{1}{N-1}} + h_{q}(A_{1})^{\frac{1}{N-1}} },
$$
for any $L > 0$ sufficiently small, and obtain that 
\begin{equation}\label{E:CDKN-1}
h_{q}(A_{1/2} )^{\frac{1}{N-1}}
		\geq 
		\sigma^{(1/2)}_{K,N-1}( A_{1} - A_{0}) 
				\left\{  h_{q}(A_{0})^{\frac{1}{N-1}} + h_{q}(A_{1})^{\frac{1}{N-1}} \right\}.
\end{equation}
Now one can observe that \eqref{E:CDKN-1} is precisely the inequality requested for $\CD^{*}_{loc}(K,N-1)$ to hold. 
As stated in Section \ref{Ss:geom}, the reduced curvature-dimension condition verifies the local-to-global property. 
In particular, see \cite[Lemma 5.1, Theorem 5.2]{cavasturm:MCP}, if a function verifies \eqref{E:CDKN-1} locally, 
then it also satisfies it globally. 
Hence $h_{q}$ also verifies the inequality requested for $\CD^{*}(K,N-1)$ to hold, i.e. for $\qq$-a.e. $q \in Q$, the density $h_{q}$ verifies \eqref{E:curvdensmm}.
\\

{\bf Step 4.}\\ 
If $N =1$ and $K >0$, $\CD^{*}(K,1)$ and $\CD_{loc}(K,1)$ are equivalent. We therefore prove the claim for $\CD_{loc}(K,1)$.
Since the condition $\CD_{loc}(K,1)$ does not depend on $K$, our argument will also be valid for $N =1$ and $K <0$. 

So repeat the same construction of {\bf Step 1.} and obtain for $\qq$-a.e. $q \in Q$
$$
(L_{s}) h_{q}( (1-s) t_{0} + s t_{1} )	\geq 		(1-s) L_{0} h_{q}( t_{0} )+  s L_{1} h_{q}( t_{1} ),
$$
for any $s \in [0,1]$ and $L_{0}$ and $L_{1}$ sufficiently small. As before, we deduce for $s = 1/2$ that
$$
\frac{L_{0} + L_{1}}{2} h_{q}( A_{1/2} ) \geq  \frac{1}{2}   \left(L_{0} h_{q}( A_{0} )+  L_{1}h_{q}( A_{1} ) \right).
$$
Now taking $L_{0} = 0$ or $L_{1} = 0$, it follows that necessarily $h_{q}$ has to be constant. 
\end{proof}

\begin{remark}
Inequality \eqref{E:curvdensmm} is the weak formulation of the following differential inequality on $h_{q,t_{0},t_{1}}$: 
\begin{equation}\label{eq:hqt0t1}
\left(h_{q,t_{0},t_{1}}^{\frac{1}{N-1}}\right)'' + (t_{1}-t_{0})^{2} \frac{K}{N-1}h_{q,t_{0},t_{1}}^{\frac{1}{N-1}} \leq 0,
\end{equation}
for all $t_{0}<t_{1} \in \dom(g(q,\cdot))$, where $h_{q,t_{0},t_{1}} (s) : = h_{q} ((1-s)t_{0} + st_{1})$. It is easy to observe that the differential inequality  \eqref{eq:hqt0t1}  on $h_{q,t_{0},t_{1}}$ is equivalent to the following differential inequality on $h_q$:
$$
\left(h_{q}^{\frac{1}{N-1}}\right)'' + \frac{K}{N-1}h_{q}^{\frac{1}{N-1}} \leq 0,
$$
that is precisely \eqref{E:CD-N-1}. Then Theorem \ref{T:CDKN-1} can be alternatively stated as follows. \\ 
\emph{
If $(X,\sfd,\mm)$ is an essentially non-branching m.m.s. verifying $\CD^{*}(K,N)$ (or $\CD_{loc}(K,N)$) 
and $\f : X \to \R$ is a 1-Lipschitz function, then the corresponding decomposition of the space 
in maximal rays $\{ X_{q}\}_{q\in Q}$ produces a disintegration $\{\mm_{q} \}_{q\in Q}$ of $\mm$ so that  for $\qq$-a.e. $q\in Q$, 
$$
\textrm{the m.m.s. }(  \dom(g(q,\cdot)), |\cdot|, h_{q} \mathcal{L}^{1}) \quad \textrm{verifies} \quad \CD(K,N).
$$
}
Accordingly, from now on we will say that the disintegration $q \mapsto \mm_{q}$ is a $\CD(K,N)$ disintegration.
\end{remark}
\bigskip

\section{Localization via Optimal Transportation}\label{S:Local}

In this section we prove the next localization result using $L^{1}$-optimal transportation theory.

\begin{theorem}\label{T:localize}
Let $(X,\sfd, \mm)$ be an essentially non-branching metric measure space verifying the $\CD_{loc}(K,N)$ condition for some $K\in \R$ and $N\in [1,\infty)$. 
Let $f : X \to \R$ be $\mm$-integrable such that $\int_{X} f\, \mm = 0$ and assume the existence of $x_{0} \in X$ such that $\int_{X} | f(x) |\,  \sfd(x,x_{0})\, \mm(dx)< \infty$. 
\medskip

Then the space $X$ can be written as the disjoint union of two sets $Z$ and $\mathcal{T}$ with $\mathcal{T}$ admitting a partition 
$\{ X_{q} \}_{q \in Q}$ and a corresponding disintegration of $\mm\llcorner_{\mathcal{T}}$, $\{\mm_{q} \}_{q \in Q}$ such that: 

\begin{itemize}
\item For any $\mm$-measurable set $B \subset \mathcal{T}$ it holds 
$$
\mm(B) = \int_{Q} \mm_{q}(B) \, \qq(dq), 
$$
where $\qq$ is a probability measure over $Q$ defined on the quotient $\sigma$-algebra $\mathcal{Q}$. 
\medskip
\item For $\qq$-almost every $q \in Q$, the set $X_{q}$ is a geodesic and $\mm_{q}$ is supported on it. 
Moreover $q \mapsto \mm_{q}$ is a $\CD(K,N)$ disintegration.
\medskip
\item For $\qq$-almost every $q \in Q$, it holds $\int_{X_{q}} f \, \mm_{q} = 0$ and $f = 0$ $\mm$-a.e. in $Z$.
\end{itemize}
\end{theorem}

\begin{proof}

{\bf Step 1.}   \\
Let $f  : X \to \R$ be such that $\int f \, \mm = 0$ and $\int_{X} | f(x) |\,  \sfd(x,x_{0})\, \mm(dx)< \infty$, for some $x_{0} \in X$.
Then consider $f_{+}$ and $f_{-}$, the positive and the negative part of $f$, respectively. It follows that if we define
$$
\mu_{0} : = \left( \int f_{+}\, \mm \right)^{-1} f_{+} \, \mm, \qquad \mu_{1} : = \left( \int f_{-}\, \mm \right)^{-1} f_{-} \, \mm
$$
then $\mu_{0}, \mu_{1}$ are both probability measures and are concentrated on two disjoint subsets of $X$, namely $\{ f_{+} > 0 \}$ and $\{ f_{-} > 0 \}$ respectively. 

The integrability condition on $f$ ensures the existence of an $L^{1}$-Kantorovich potential for $\mu_{0}$ and $\mu_{1}$, 
i.e. a map $\f : X \to \R$ with global Lipschitz constant equals 1 
such that any transport plan 
$\pi \in \mathcal{P}(X \times X)$ with marginals $\mu_{0}$ and $\mu_{1}$ is optimal for the $L^{1}$-transportation distance if and only if $\pi(\Gamma) =1$, 
where
$$
\Gamma : = \{ (x,y) \in X\times X \colon \f(x) - \f(y) = \sfd(x,y) \}.
$$
Here with global Lipschitz constant we mean 
$$
\| \f \|_{lip} : = \sup_{x \neq y } \frac{|\f(x) - \f(y)|}{\sfd(x,y)},
$$
and for $\pi$ to have marginal measures $\mu_{0}$ and $\mu_{1}$ means that 
$$
 (P_{1})_\sharp \, \pi = \mu_{0}, \qquad   (P_{2})_\sharp\, \pi = \mu_{1}, 
$$
where $P_{i} : X \times X \to X$ is the projection on the $i$-th component, with $i =1,2$. 
For the existence of $\f$ just observe that the dual optimal transportation problem reduces to find the maximizer of 
$\int_{X} f \phi \, \mm$ among all $\phi : X \to \R$ that are $1$-Lipschitz whose existence follows straightforwardly from the integrability condition on $f$.

Then from the Theorem \ref{T:CDKN-1} applied to $\f$ and $\Gamma$ we have that 
$$
\mm\llcorner_{\mathcal{T}} = \int_{Q} \mm_{q} \, \qq(dq), 
$$
and $q \mapsto \mm_{q}$ is a  $\CD(K,N)$ disintegration. 
\medskip

{\bf Step 2.} \\
It remains  to show that for this particular choice of $\f$, the last 
part of Theorem \ref{T:localize} holds. Therefore let $\pi$ be any optimal transport plan between $\mu_{0}$ and $\mu_{1}$, i.e. $\pi(\Gamma) = 1$.

Suppose by contradiction the existence of a measurable set $A \subset X \setminus \mathcal{T}$ with $\mm(A) > 0$ and $f(x) \neq 0$ for all $x \in A$.
Then we can assume with no loss of generality that $\mu_{0}(A) > 0$ and $A \subset X \setminus \mathcal{T}_{e}$. Recall that $\mathcal{T}_{e}$ is the transport set 
with end points.

Since $\mu_{0}$ and $\mu_{1}$ are concentrated on two disjoint sets, any optimal plan is indeed concentrated on $\Gamma \setminus \left\{ x = y \right\}$. 
Then 
\begin{align*}
\mu_{0}(\mathcal{T}_{e}) 	= 	&~	\pi \big( \left(\Gamma \setminus \{x = y \}\right) \cap \mathcal{T}_{e}\times X \big) \crcr
					\geq 	&~ 	\pi \big( \left\{ (x,y) \in \Gamma \setminus \{x=y \} \colon \exists\,z\in X, \, z \neq x, \, (x,z)\in \Gamma \right\} \big) \crcr
					= 	&~	\pi (\Gamma \setminus \{ x= y \} ) = 1. 
\end{align*}
Since $\mu_{0}$ is a probability measure, this is in contradiction with $\mu_{0}(A) >0$. Therefore $f(x) = 0$ for a.e. $x \in X \setminus \mathcal{T}$. 
\medskip

{\bf Step 3.}
It remains to show that $\qq$-a.e. $q \in Q$ one has $\int f \, \mm_{q} = 0$. Since for both $\mu_{0}$ and $\mu_{1}$ the set $\mathcal{T}_{e} \setminus \mathcal{T}$ is negligible,
for any Borel set $C \subset Q$ 
\begin{eqnarray}
\mu_{0}(\QQ^{-1}(C)) 	&= & \pi \Big(  (\QQ^{-1}(C) \times X)  \cap \Gamma \setminus \{ x = y\} \Big) \nonumber \\
					&= & \pi \Big(  ( X \times \QQ^{-1}(C))  \cap \Gamma \setminus \{ x = y\} \Big) \nonumber \\
					&= &  \mu_{1}(\QQ^{-1}(C)),  \label{eq:mu0=mu1}
\end{eqnarray}
where the second equality follows from the fact that $\mathcal{T}$ does not branch: indeed since $\mu_{0}(\mathcal{T}) = \mu_{1}(\mathcal{T}) = 1$,  
then $\pi \big( (\Gamma \setminus \{ x= y\}) \cap \mathcal{T} \times \mathcal{T}  \big) =1$ and therefore if $x,y \in \mathcal{T}$ and $(x,y)\in \Gamma$, 
then necessarily $\QQ(x) = \QQ(y)$, that is they belong to the same ray. It follows that 
$$
(\QQ^{-1}(C) \times X)  \cap (\Gamma \setminus \{ x = y\}) \cap (\mathcal{T} \times \mathcal{T}) = 
( X\times \QQ^{-1}(C) )  \cap (\Gamma \setminus \{ x = y\}) \cap (\mathcal{T} \times \mathcal{T}),
$$
and \eqref{eq:mu0=mu1} follows.

Since $f$ has null mean value it holds $\int_X f_{+}(x) \mm(dx)= - \int_X f_{-}(x) \mm(dx)$, 
which combined with \eqref{eq:mu0=mu1} implies that for each Borel $C \subset Q$ 
\begin{align*}
\int_{C}  \int_{X_{q}} f(x) \mm_{q}(dx) \qq(dq) 	= &~ \int_{C}  \int_{X_{q}} f_{+}(x) \mm_{q}(dx) \qq(dq) - \int_{C}  \int_{X_{q}} f_{-}(x) \mm_{q}(dx) \qq(dq) \crcr
									= &~ \left( \int_{X} f_{+}(x) \mm(dx) \right)^{-1} \left( \mu_{0}(\QQ^{-1}(C)) - \mu_{1}(\QQ^{-1}(C))  \right)  \crcr 
									= &~ 0.
\end{align*}
Therefore for $\qq$-a.e. $q \in Q$ the integral $\int f \, \mm_{q}$ vanishes and the claim follows.
%

%
\end{proof}

\section{Sharp and rigid Isoperimetric Inequalities}\label{S:Isop}

The goal of the paper is to compare the isoperimetric profile of a m.m.s. satisfying synthetic 
Ricci lower curvature bounds with model spaces on the real line. So, in  order to start, in  the next subsection we will focus on the case $(X,\sfd)=(\R, |\cdot|)$.

\subsection{Isoperimetric profile for m.m.s. over $(\R, |\cdot|)$}
Given $K\in \R, N\in[1,+\infty)$ and $D\in (0,+\infty]$, consider the following family of probability measures

\begin{eqnarray}
\mathcal{F}^{s}_{K,N,D} : = \{ \mu \in \mathcal{P}(\R) : &\supp(\mu) \subset [0,D], \, \mu = h_{\mu} \mathcal{L}^{1},\,
h_{\mu}\, \textrm{verifies} \, \eqref{E:curvdensmm} \ \textrm{and is continuous if } N\in (1,\infty), \nonumber \\ 
& \quad h_{\mu}\equiv \textrm{const} \text{ if }N=1   \},
\end{eqnarray}
and  the corresponding  comparison \emph{synthetic} isoperimetric profile:  
$$
\mathcal{I}^{s}_{K,N,D}(v) : = \inf \left\{ \mu^{+}(A) \colon A\subset \R, \,\mu(A) = v, \, \mu \in \mathcal{F}^{s}_{K,N,D}  \right\},
$$
where $\mu^{+}(A)$ denotes the Minkowski content defined  in  \eqref{eq:MinkCont}.
\\
The term synthetic refers to $\mu \in \mathcal{F}^{s}_{K,N,D}$ meaning that the Ricci curvature bound is satisfied in its synthetic formulation:
if $\mu = h \cdot \mathcal{L}^{1}$, then $h$ verifies \eqref{E:curvdensmm}.

\medskip

The goal of this short section is to  prove that $\mathcal{I}^{s}_{K,N,D}$ coincides with  its smooth counterpart $\mathcal{I}_{K,N,D}$ defined by
\begin{equation}\label{defcI}
\mathcal{I}_{K,N,D}(v) : = \inf \left\{ \mu^{+}(A) \colon A\subset \R, \,\mu(A) = v, \, \mu \in \mathcal{F}_{K,N,D}  \right\},
\end{equation}
where now $\mathcal{F}_{K,N,D}$ denotes the set of $\mu \in \mathcal{P}(\R)$ such that  $\supp(\mu) \subset [0,D]$   
and $\mu = h \cdot \mathcal{L}^{1}$ with $h \in C^{2}((0,D))$ satisfying
\begin{equation}\label{eq:DiffIne}
\left( h^{\frac{1}{N-1}} \right)'' + \frac{K}{N-1} h^{\frac{1}{N-1}} \leq 0 \quad \text{if }N \in (1,\infty), \quad h\equiv \textrm{const} \quad \text{if }N=1.
\end{equation}

\begin{remark}\label{rem:IKNDinf}
Our notation is consistent, in the sense that the model isoperimetric profile for smooth densities $\mathcal{I}_{K,N,D}$ defined in \eqref{defcI} coincides with the model profile   $\mathcal{I}_{K,N,D}$ defined in Section \ref{SS:IKND}; for the proof  see \cite[Theorem 1.2, Corollary 3.2]{Mil}. 
 \end{remark}
\noindent
It is easily verified that $\mathcal{F}_{K,N,D} \subset \mathcal{F}^{s}_{K,N,D}$.  Also here the diameter $D$ of the support of the measure $\mu$ can attain the value $+\infty$.

In order  to prove that $\mathcal{I}_{K,N,D}(v)=\mathcal{I}^s_{K,N,D}(v)$ for every $v \in [0,1]$ the following approximation result will play a key role. 
In order to state it let us recall that a standard mollifier in $\R$ is a non negative $C^\infty(\R)$ 
function $\psi$ with compact support in $[0,1]$ such  that $\int_{\R} \psi = 1$. 
\begin{lemma}\label{lem:approxh}
Let  $D \in (0,\infty)$ and let  $h:[0,D] \to [0,\infty)$ be a continuous function. Fix $N\in (1,\infty)$ and for $\ve>0$ define
\begin{equation}
h_{\ve}(t):=[h^{\frac{1}{N-1}}\ast \psi_{\ve} (t)]^{N-1}  := \left[ \int_{\R} h(t-s)^{\frac{1}{N-1}}  \psi_{\ve} (s) \, d s\right]^{N-1} 
										=  \left[ \int_{\R} h(s)^{\frac{1}{N-1}}  \psi_{\ve} (t-s) \, d s\right]^{N-1},
\end{equation}
where $\psi_\ve(x)=\frac{1}{\ve} \psi(x/\ve)$ and $\psi$ is a standard mollifier function. The following properties hold:
\begin{enumerate}
	\item $h_{\ve}$ is a non-negative $C^\infty$ function with support in $[-\ve, D+\ve]$; \medskip
	\item $h_{\ve}\to h$ uniformly  as $\ve \downarrow 0$, in particular $h_{\ve} \to h$ in $L^{1}$. \medskip
	\item If $h$ satisfies the convexity condition \eqref{E:curvdensmm} corresponding to the above fixed $N>1$ 
		and some $K \in \R$ then also $h_{\ve}$ does. In particular $h_{\ve}$ satisfies the differential inequality \eqref{eq:DiffIne}.
\end{enumerate}
\end{lemma}

\begin{proof}
First of all observe that since the mollifier function $\psi_{\ve}$ is non negative, then the mollification preserves the order, i.e.
\begin{equation}\label{eq:f>g}
f (t) \leq  g (t) \quad \text{for a.e. t} \in \R \quad \Rightarrow \quad  f_{\ve}(t) \leq g_{\ve}(t) \quad \text{for every } t \in \R.
\end{equation}
The proofs of the first and second claims   follow by the standard properties of convolution, for which we refer to \cite[Theorem 6, Appendix C.4]{Evans:PDEs}.
%

In order to get the third claim we use \eqref{E:curvdensmm} together with the fact that $\psi_\varepsilon$ is non-negative  to infer that  for every fixed $t_0, t_1 \in [0,D]$ and $s\in [0,1]$ the following holds
\begin{eqnarray}
h_{\ve}((1-s) t_0+s t_1)^{\frac{1}{N-1}}	& 	= 	& ( h^{\frac{1}{N-1}}\ast \psi_{\ve}) ((1-s) t_0+s t_1) \nonumber \\
                                                                   &      =      & \int_{\R}  h^{\frac{1}{N-1}} ((1-s)\, (t_0-t) +s \, (t_1 -t)) \;   \psi_{\ve}(t) \, d t \nonumber \\
                                                                   & 	\geq	&  \sigma^{(1-s)}_{K,N-1}(t_{1} - t_{0})    \int_{\R}  h^{\frac{1}{N-1}} (t_0-t) \;   \psi_{\ve}(t) \, d t + \sigma^{(s)}_{K,N-1}(t_{1} - t_{0})   \int_{\R}  h^{\frac{1}{N-1}} (t_1-t) \;   \psi_{\ve}(t) \, d t  \nonumber \\
                                                           	&	=	&  \sigma^{(1-s)}_{K,N-1}(t_{1} - t_{0}) h_{\ve} (t_{0})^{\frac{1}{N-1}} 
											+ \sigma^{(s)}_{K,N-1}(t_{1} - t_{0}) h_{\ve} (t_{1})^{\frac{1}{N-1}}.
\end{eqnarray}
It is finally a standard computation to check, for $C^2$ functions, that the convexity inequality \eqref{E:curvdensmm} is equivalent to the differential inequality \eqref{eq:DiffIne}.
\end{proof}

\medskip

\begin{theorem}\label{thm:I=Is}
For every $v\in [0,1]$, $K \in \R$, $N\in [1,\infty)$, $D\in (0,\infty]$ it holds $\mathcal{I}^{s}_{K,N,D}(v)=\mathcal{I}_{K,N,D}(v)$.
\end{theorem}

\begin{proof}
For $N=1$ the statement is trivial as   $\mathcal{F}_{K,1,D} = \mathcal{F}^{s}_{K,1,D}$, so we can assume $N \in (1,\infty)$.
\\It is also clear that  $\mathcal{I}^{s}_{K,N,D}(v)\leq \mathcal{I}_{K,N,D}(v)$ for every $v\in [0,1]$, since   $\mathcal{F}_{K,N,D} \subset \mathcal{F}^{s}_{K,N,D}$.
\\In order to show the converse inequality,  it is enough to consider the case $D\in (0,\infty)$: indeed  for $K>0$ we know that the diameter of the space is bounded  by the Bonnet-Myers Theorem and for  $K\leq 0, D=\infty$ it holds  $\mathcal{I}_{K,N,D}(v)\equiv 0$.
\\For an arbitrary measure $\mu=h \cdot \mathcal{L}^{1} \in  \mathcal{F}^{s}_{K,N,D}$,  
%
%
Lemma \ref{lem:approxh} gives a sequence $h_{k}\in C^{\infty}(\R)$ such that 
$$
\supp (h_k)\subset \left[ -\frac{1}{k}, D+\frac{1}{k} \right], 
	\qquad   \mu_k:=h_k \cdot \mathcal{L}^{1} \in  \mathcal{F}_{K,N,D+\frac{2}{k}}, 
		\qquad  \|h_{k} - h\|_{L^{1}((0,D))} \longrightarrow 0.
$$
Therefore the measures $\mu_{k}$ converge to $\mu$ in total variation sense:
$$
\lim_{k \to \infty}\| \mu_{k} - \mu \|_{TV}  = \lim_{k\to \infty} \sup \big\{ |\mu_{k}(A) - \mu(A)| \colon  A\subset \R \ \textrm{Borel} \big\} = 0.
$$
At this point we can repeat verbatim the proof of \cite[Proposition 6.1]{Mil} to get 
$$
\mathcal{I}^{s}_{(\R,|\cdot|,\mu)}(v) \geq \limsup_{k} \mathcal{I}_{(\R,|\cdot|,\mu_k)}(v) \geq  \mathcal{I}_{K,N,D}(v), \quad \text{for every } v \in [0,1].
$$
%
\end{proof}

\medskip

\subsection{Sharp lower bounds for the isoperimetric profile}

The goal of this section is to prove the following result, which is the heart of the present work.

\begin{theorem}\label{T:iso}
Let $(X,\sfd,\mm)$ be a metric measure space with $\mm(X)=1$, verifying  the essentially non-branching property and $\CD_{loc}(K,N)$ for some $K\in \R,N \in [1,\infty)$.
Let $D$ be the diameter of $X$, possibly assuming the value $\infty$.
\medskip

Then for every $v\in [0,1]$, 
$$
\cI_{(X,\sfd,\mm)}(v) \ \geq \ \cI_{K,N,D}(v), 
$$
where $\cI_{K,N,D}$ is the model isoperimetric profile defined in \eqref{defcI}.
\end{theorem}

\begin{proof}
First of all we can assume $D<\infty$ and therefore $\mm \in \mathcal{P}_{2}(X)$: indeed from the Bonnet-Myers Theorem if $K>0$ then $D<\infty$, and if $K\leq 0$ and $D=\infty$ then the model isoperimetric profile \eqref{defcI} trivializes, i.e. $\cI_{K,N,\infty}\equiv 0$ for $K\leq 0$.

For $v=0,1$ one can take as competitor the empty set and the whole space respectively, so it trivially holds 
$$
\cI_{(X,\sfd,\mm)}(0)=\cI_{(X,\sfd,\mm)}(1)= \cI_{K,N,D}(0)=\cI_{K,N,D}(1)=0.
$$
Fix then $v\in(0,1)$ and let $A\subset X$  be an arbitrary Borel subset of $X$ such that $\mm(A)=v$.
Consider the $\mm$-measurable function $f(x) : = \chi_{A}(x)  - v$ and notice that  $\int_{X} f \, \mm = 0$. 
Thus $f$ verifies the  hypothesis of Theorem \ref{T:localize} and noticing that $f$ is never null, 
we can decompose $X = Y \cup \mathcal{T}$ with 
$$
\mm(Y)=0, \qquad   \mm\llcorner_{\mathcal{T}} = \int_{Q} \mm_{q}\, \qq(dq), 
$$
with $\mm_{q} = g(q,\cdot)_\sharp \left( h_{q} \cdot \mathcal{L}^{1}\right)$; 
moreover,  for $\qq$-a.e. $q \in Q$,  the density $h_{q}$ verifies \eqref{E:curvdensmm}  and 
$$
\int_{X} f(z) \, \mm_{q}(dz) =  \int_{\dom(g(q,\cdot))} f(g(q,t)) \cdot h_{q}(t) \, \mathcal{L}^{1}(dt) = 0.
$$
Therefore 
\begin{equation}\label{eq:volhq}
v=\mm_{q} ( A \cap \{ g(q,t) : t\in \R \} ) = (h_{q}\mathcal{L}^1) (g(q,\cdot)^{-1}(A)), \quad \text{ for $\qq$-a.e. $q \in Q$}. 
\end{equation}
For every $\ve>0$ we  then have  
\begin{align*}
\frac{\mm(A^\ve)-\mm(A)}{\ve} 	&~  =  \frac{1}{\ve} \int_{\mathcal{T}} \chi_{A^\ve\setminus A} \,\mm(dx) 
									=  \frac{1}{\ve} \int_{Q} \left( \int_{X}   \chi_{A^\ve\setminus A} \, \mm_{q} (dx) \right)\, \qq(dq) \crcr
						&~ =    \int_{Q} \frac{1}{\ve} \left( \int_{\dom(g(q,\cdot))}  \chi_{A^\ve\setminus A} \,h_{q}(t) \, \mathcal{L}^{1}(dt) \right)\, \qq(dq) \crcr
						&~ =    \int_{Q} \left(  \frac{(h_{q}\mathcal{L}^1)(g(q,\cdot)^{-1}(A^\ve))-  (h_{q}\mathcal{L}^1)(g(q,\cdot)^{-1}(A))}{\ve}  \right)\, \qq(dq) \crcr
						&~ \geq    \int_{Q} \left(  \frac{(h_{q}\mathcal{L}^1)((g(q,\cdot)^{-1}(A))^\ve)-  (h_{q}\mathcal{L}^1)(g(q,\cdot)^{-1}(A))}{\ve}  \right)\, \qq(dq), \crcr
\end{align*}
where the last inequality is given by the inclusion $ (g(q,\cdot)^{-1}(A))^\ve \cap \supp(h_q) \subset  g(q,\cdot)^{-1}(A^\ve)$. \\
Recalling \eqref{eq:volhq} together with $h_{q}\mathcal{L}^1\in \mathcal{F}^{s}_{K,N,D}$, by Fatou's Lemma we get
\begin{align*}
\mm^+(A)		&~ = 	\liminf_{\ve\downarrow 0} \frac{\mm(A^\ve)-\mm(A)}{\ve} \crcr
			&~\geq 	\int_{Q}  \left(  \liminf_{\ve\downarrow 0} \frac{(h_{q}\mathcal{L}^1)((g(q,\cdot)^{-1}(A))^\ve) -  
											(h_{q}\mathcal{L}^1)(g(q,\cdot)^{-1}(A))}{\ve}  \right)\, \qq(dq) \crcr
			&~ =   	\int_{Q} \left( (h_{q}\mathcal{L}^1)^+(g(q,\cdot)^{-1}(A))  \right)\, \qq(dq) \crcr
			&~ \geq  	\int_{Q}  \cI^s_{K,N,D} (v) \, \qq(dq)  \crcr
			&~ =  	\cI_{K,N,D} (v),
\end{align*}
where in the last equality we used Theorem \ref{thm:I=Is}. The conclusion follows from Remark \ref{rem:IKNDinf}.
\end{proof}

\medskip
\noindent
{\textbf{Proof of Theorems  \ref{thm:LG} and \ref{thm:mainIsoComp}}}
Since $\RCD^*(K,N)$-spaces are essentially non branching (see \cite{RS2014}) and the $\CD^*(K,N)$ condition is equivalent to $\CD_{loc}(K,N)$ for $N\in(1,\infty)$ and for $K\geq 0, N=1$, then we can apply  Theorem \ref{T:iso} and get Theorem \ref{thm:mainIsoComp}.
As already observed in the introduction, the Levy-Gromov isoperimetric inequality claimed in Theorem \ref{thm:LG} is just a special case of Theorem \ref{thm:mainIsoComp} when $K>0$ and $N$ is a positive integer.
\hfill$\Box$

\medskip

\subsection{Rigidity in the isoperimetric comparison estimates: proof of Theorem \ref{thm:Rigidity}}

The following  lemma will play a key role for proving the rigidity and the almost rigidity statements.

\begin{lemma}\label{lem:D<inf}
For every  $v\in (0,1), N>1$ and $\ve_0\in (0,\pi)$ there exist $\eta=\eta(v,N,\ve_0)>0$ such that for every $\delta\in \left[0, \frac{N-1}{2}\right]$ and for every $D\in(0,\pi-\ve_0)$ it holds
\begin{equation}\label{eq:D<infquant}
\cI_{N-1-\delta, N+\delta,D}(v) \geq \cI_{N-1-\delta, N+\delta,\infty}(v)+\eta.
\end{equation}
\end{lemma}

\begin{proof}
Fix $v\in (0,1), N>1$ and $\ve_0\in (0,\pi)$ as above. First of all it is not difficult to see that 
$$\lim_{D\downarrow 0}  \cI_{N-1-\delta, N+\delta,D}(v) \to +\infty \quad \text{uniformly for }\delta \in \left[0,\frac{N-1}{2}\right]. $$
Therefore, in order to establish \eqref{eq:D<infquant}, it is enough to consider the case $D\in (\ve_1,\pi-\ve_0)$ for some  $\ve_1=\ve_1(N)\in \big(0, \pi-2\ve_0\big)$ independent of $\delta\in \left[0,\frac{N-1}{2}\right]$.

By \cite{Mil} we know that there exist  $A=A_{\delta,N,D} \subset [0,D]$ and  $\mu_{N-1-\delta,N+\delta,D}\in {\mathcal F}_{N-1-\delta,N+\delta,D}$  such that  
$$
\mu_{N-1-\delta,N+\delta,D}(A)= v \quad \text{and} \quad \mu^{+}_{N-1-\delta,N+\delta,D}(A)= \cI_{N-1-\delta,N+\delta,D}(v),
$$
where the minimizer $\mu_{N-1-\delta,N+\delta,D}$ is given by
$$
\mu_{N-1-\delta,N+\delta,\infty} \llcorner_{[0,D]} =\lambda \, \mu_{N-1-\delta,N+\delta,D}, \quad \text{for some } \lambda=\lambda_{\delta,v,N} \in [\ve_2, 1-\ve_2] \subset (0,1),$$
for some $\ve_2={\ve_2}(\ve_0,\ve_1)\in (0,1/2)$.
Observe that 
$$\mu_{N-1-\delta,N+\delta,\infty}(A)=\lambda v \quad  \text{and} \quad \mu^+_{N-1-\delta,N+\delta,\infty}(A)=\lambda \,  \mu^+_{N-1-\delta,N+\delta,D}(A), $$  
and that  the maps  
$$\cI_{N-1-\delta,N+\delta,\infty}:  [\ve_2, 1-\ve_2] \to \R^+, \quad t \mapsto \cI_{N-1-\delta,N+\delta,\infty}(t v)$$ 
are strictly concave functions uniformly with respect to $\delta\in \big[0, \frac{N-1}{2}\big]$. Since   $\cI_{N-1-\delta,N+\delta,\infty}(0)=0$, it follows that there exists $\eta=\eta(v,N,\ve_0)>0$ such that 
\begin{eqnarray}
\lambda \, \cI_{N-1-\delta,N+\delta,\infty}(v) & \leq &  \cI_{N-1-\delta,N+\delta,\infty}(\lambda \, v) - \eta \leq \mu^+_{N-1-\delta,N+\delta,\infty}(A) - \eta \nonumber \\
                                                                     & = &  \lambda  \, \mu^+_{N-1-\delta,N+\delta,D}(A) - \eta =\lambda \, \cI_{N-1-\delta,N+\delta,D}(v)- \eta. \nonumber
\end{eqnarray}
Multiplying both sides by  $\lambda^{-1}\in \big[ \frac{1}{1-\ve_2},  \frac{1}{\ve_2} \big]$ we obtain the thesis. 
\end{proof}

%
%

%
%
%
%
%
%

\medskip
\noindent
{\textbf{Proof of Theorem \ref{thm:Rigidity}}}.

First of all we claim that if for some $\bar{v} \in (0,1)$ one has $\cI_{(X,\sfd,\mm)}(\bar{v})=\cI_{N-1,N,\pi}(\bar{v})$ 
then $(X,\sfd)$ has diameter equal to $\pi$; then the Maximal Diameter Theorem \cite[Theorem 1.4]{Ket} 
will imply that $X$ is a spherical suspension over an $\RCD^*(N-2,N-1)$ space $Y$ as desired.   

So suppose by contradiction $(X,\sfd)$ has diameter equal to $\pi-\ve_0<\pi$ then by Lemma \ref{lem:D<inf}  there exists $\delta>0$ such that
\begin{equation}  \nonumber
\cI_{N-1,N,\pi}(\bar{v}) \leq  \cI_{N-1,N,D}(\bar{v}) -\delta, \quad \forall D\in (0,\pi-\ve_0].
\end{equation}
At this point we could already conclude by observing that  we reached the contradiction 
$$\cI_{N-1,N,\pi}(\bar{v})=\cI_{(X,\sfd,\mm)}(\bar{v})\geq  \cI_{N-1,N, \pi-\varepsilon_0}(\bar{v}) \geq \cI_{N-1,N,\pi}(\bar{v})  + \delta,$$ 
where in the first inequality we applied Theorem  \ref{T:iso}. 
Let us also give a more direct argument using 1-d localization.  Let $A\subset X$  be such that 
\begin{equation} \nonumber
\mm(A)=\bar{v} \quad \text{and} \quad \mm^+(A)\leq \cI_{(X,\sfd,\mm)}(\bar{v})+\frac{\delta}{2}=\cI_{N-1,N,\pi}(\bar{v})+\frac{\delta}{2}.
\end{equation}
Repeating the proof of Theorem  \ref{T:iso}, we obtain the contradiction
\begin{eqnarray}\label{eq:PfRig0}
\cI_{N-1,N,\pi}(\bar{v})+\frac{\delta}{2} &\geq&  \mm^+(A)  \geq  \int_{Q} \left( (h_{q}\mathcal{L}^1)^+(g(q,\cdot)^{-1}(A))  \right)\, \qq(dq)  
															\geq  \int_{Q}  \cI_{N-1,N, |\supp(h_q)|} (\bar{v}) \, \qq(dq) \nonumber \\
                                                                 &\geq&   \cI_{N-1,N,\pi}(\bar{v})+\delta, \nonumber
\end{eqnarray}
where  $|\supp(h_q)|$ denotes the length of the segment $\supp(h_q)\subset \R$ and we made use that, since by Theorem \ref{T:localize} we know that $\supp(h_q)$ is isometric to  a geodesic $X_q$ of $(X,\sfd)$ for $\qq$-a.e. $q$, then  $|\supp(h_q)| \leq \pi-\ve_0$. 

\medskip
This concludes the first part of the proof. We now proceed to characterize the isoperimetric sets.

\medskip

{\textbf{Step 1.}} \\
If there exists a Borel subset $A\subset X$ with $\mm(A)=\bar{v}$ such that $\mm^+(A)= \cI_{(X,\sfd,\mm)}(\bar{v})=\cI_{N-1,N,\pi}(\bar{v})$ 
then we have just proved that $(X,\sfd,\mm)$ is a spherical suspension, i.e. $X\simeq [0,\pi]\times^{N-1}_{\sin}Y$.  \\ 
Now we claim that the following more precise picture holds: 
\begin{enumerate}
\item   $(Y, \sfd_Y, \mm_Y)$ is an   $\RCD^*(N-2,N-1)$ space  and $(Q,\qq)$ is isomorphic as measure space to $(Y,\mm_Y)$;
\item    for $\qq$-a.e. $q$ it holds  $h_q(t)= c\,  (\sin t)^{N-1}$, where $c>0$ is a normalizing constant.
\end{enumerate}

Indeed, consider the 1-Lipschitz function  $\f$  inducing the $1$-d localization associated to $A$.  By the discussion right before Step 1 we know 
that for  $\qq$-a.e. $q \in Q$ the ray $X_q$ has length $\pi$ and $\mm_{q}^{+}(A\cap X_{q}) = \cI_{N-1,N,\pi}(\bar{v})$.  \\
Let us now fix one of those rays $X_{q}$ and call $N,S \in X_q$  the  endpoints of the geodesic $X_q$. Then $\sfd(S,N)=\textrm{length}(X_q)=\pi$ and by the Maximal Diameter Theorem \cite{Ket} $X$ is a spherical suspension, i.e. $X\simeq [0,\pi]\times^{N-1}_{\sin}Y$ for some $\RCD^*(N-2,N-1)$ space $(Y,\sfd_Y,\mm_Y)$    and $N,S$ correspond respectively to the north and south pole of such a suspension structure that is $S=(0,y_0)$, $N=(\pi,y_0)$, for some $y_0 \in Y$. 
\\For the rest of the proof we will identify $X$ with $[0,\pi]\times^{N-1}_{\sin}Y$ and with a slight abuse of notation we will write $J\times E$ meaning $\{(t,y)\in  [0,\pi]\times^{N-1}_{\sin}Y \,: \, t \in J \text{ and } y \in E\}$.  By the choice of $N,S$ we infer that
$$
\f(S) - \f(N) = \sfd(S,N) = \pi, 
$$
and with no loss of generality, just adding a constant to $\f$, we can assume $\f(S) = \pi$ and $\f(N)= 0$.
\\ Take now any other element $(t,y) \in [0,\pi]\times^{N-1}_{\sin}Y$. Since the curve $[0,\pi] \ni t \mapsto (t,y)$ is a geodesic from $S$ to $N$, from Lemma \ref{L:cicli} it follows that 
$$
\pi - \f((t,y))= \f(S) - \f((t,y)) = \sfd(S, (t,y)) = t. 
$$
Therefore for any $(t,y) \in [0,\pi]\times^{N-1}_{\sin}Y$ it holds $\f((t,y)) = \pi - t$.
It follows that for $\qq$-a.e. $q \in Q$ there exists $y_q\in Y$ such that $X_q=[0,\pi]\times \{y_q\}$.
We deduce that $\T = (0,\pi) \times  Y$ and $\QQ((t,y)) = (1/2, y)$ is a quotient map yielding $(Q, \qq) \simeq (Y,\mm_{Y})$ as measure spaces; in particular claims (1) and (2) are proved.

\medskip  
  
{\textbf{Step 2.}} \\
Let $A\subset X$ be as in Step 1. Called $\mu_{N-1,N}:=  (\sin (t))^{N-1} \, {\mathcal L}^1 \llcorner [0,\pi]$, we claim that there exists  a subinterval $I_{\bar v} \subset [0,\pi]$ with $ \mu_{N-1,N}(I_{\bar v})=\bar{v}$ such that $A=I_{\bar{v}}\times Y$. Recall that  $A=I_{\bar{v}}\times Y$ has to be intended in the coordinates $(t,y)$ of $[0,\pi] \times^{N-1}_{\sin}Y$; in other words it is not a product as m.m.s. but just  a short-hand notation we are using to denote the set
$$
 A=\{(t,q) \in  [0,\pi]\times^{N-1}_{\sin}Y \, : \, t\in I_{\bar{v}}\}.
$$
In order to prove such a claim we first recall that by \cite{Mil} (this is actually a classical result going back to L\'evy and Gromov at least for integer $N$) 
 there exists a  Borel set $J_{\bar{v}}\subset [0,\pi]$ with $\mu_{N-1,N}(J_{\bar{v}})={\bar{v}}$ such that  
$\mu^+_{N-1,N}(J_{\bar{v}})=\cI_{N-1,N,\infty}({\bar{v}})$, and such a Borel set must be an interval either of the form $[0,r_{\bar{v}}]$ or $[\pi-r_{\bar{v}}, \pi]$ for a suitable $r_{\bar{v}} \in (0,\pi)$. 
\\By the previous discussion, for $\qq$-a.e. $q\in Q$,  we know that $g(q,\cdot)^{-1}(A)$ must be either equal to $[0,r_{\bar{v}}]$ or to $[\pi-r_{\bar{v}}, \pi]$. But now the configuration where both
$$
\qq(\{q\,:\, g(q,\cdot)^{-1}(A)= [0,r_{\bar{v}}]\})>0 \quad \text{and} \quad \qq(\{q\,:\, g(q,\cdot)^{-1}(A)= [\pi-r_{\bar{v}},\pi]\})>0,
$$
creates an interface between the two corresponding subsets of $A$ which will cost a higher Minkowski content than the 
configuration where either $g(q,\cdot)^{-1}(A)=[0,r_{\bar{v}}]$ for $\qq$-a.e. $q$ or  $g(q,\cdot)^{-1}(A)=[\pi-r_{\bar{v}}, \pi]$ for $\qq$-a.e. $q$.  \\
Let us give a rigorous proof of the last intuitive statement.  Assume by contradiction that there exist $Q_1,Q_2$ Borel subsets of  $Y$ with $\mm_Y(Q_1)=1-\mm_{Y}(Q_2)\in (0,1)$,  such that $A=A_1\cup A_2$ where 
$$
A_{1} : = [0,r_{\bar{v}}]\times Q_1, \quad  A_{2}=[\pi - r_{\bar{v}}, \pi] \times Q_2.
$$
Calling  $A = A_{1} \cup A_{2}$, clearly $\mm(A)=\bar{v}$.  Suppose by contradiction that $\mm^+(A)= \cI_{(X,\sfd,\mm)}(\bar{v})=\cI_{N-1,N,\pi}(\bar{v})$. 
Notice that if $(t,p), (t,q) \in [0,\pi]\times^{N-1}_{\sin} Y$ then their distance $\sfd((t,p), (t,q)) = \sin(t)^{N-1} \sfd_{Y}(p,q) \leq \sfd_{Y}(p,q)$.
Therefore  
\begin{equation} \label{eq:A1eps}
A_1^\ve\supset [0,r_{{\bar{v}}} + \ve]\times Q_1 \cup  [0,r_{{\bar{v}}}] \times Q_1^{\ve},
\end{equation}
 with analogous inclusion for $A_{2}^{\ve}$.
%
Using \eqref{eq:A1eps} it is not difficult to check that, thanks to the symmetry of the space, it holds
\begin{align}\label{eq:Ave-A}
\mm(A^\ve)-\mm(A) \geq &~ \mm\big([r_{\bar{v}}, r_{\bar{v}}+\ve]\times  Q_1\big) +  \mm\big([\pi-r_{\bar{v}}-\ve, \pi-r_{\bar{v}}] \times  Q_2 \big) \nonumber \\
&  +  \sum_{i = 1,2} \mm\big( [0, \min\{r_{\bar{v}}, \pi-r_{\bar{v}}\}] \times (Q_i^\ve\setminus Q_i)\big) .
\end{align}
Along the same lines of the proof of Theorem  \ref{T:iso} and using that from Step 1 we know $(Y,\mm_Y)\simeq(Q,\qq)$, it is not hard to show that
\begin{eqnarray}
\liminf_{\ve\downarrow 0} \frac{ \mm\big([r_{\bar{v}}, r_{\bar{v}}+\ve]\times  Q_1\big)}{\ve}& = & \qq(Q_1) \, \cI_{N-1,N,\pi}(\bar{v})=\mm_Y(Q_1)   \, \cI_{N-1,N,\pi}(\bar{v}),  \label{eq:mYQ1}  \\
\liminf_{\ve\downarrow 0} \frac{ \mm\big([\pi-r_{\bar{v}}-\ve, \pi-r_{\bar{v}}] \times  Q_2 \big)}{\ve}&= & \qq(Q_2) \, \cI_{N-1,N,\pi}(\bar{v})=\mm_Y(Q_2)   \, \cI_{N-1,N,\pi}(\bar{v}) .\label{eq:mYQ2}
\end{eqnarray}
Moreover, since  from Step 1 for  $\qq$-a.e. $q$ it holds  $h_q(t)= c\,  (\sin t)^{N-1}$ and  $(Y,\mm_Y)\simeq(Q,\qq)$, we also get
\begin{eqnarray}
\liminf_{\ve\downarrow 0} \frac{\mm\big( [0, \min\{r_{\bar{v}}, \pi-r_{\bar{v}}\}] \times (Q_i^\ve\setminus Q_i)\big)}{\ve}
&=&  \liminf_{\ve\downarrow 0}  \frac{1}{\ve} \int_{Q_i^\ve\setminus Q_i}  \left[ c \int_0^{ \min\{r_{\bar{v}}, \pi-r_{\bar{v}}\}} (\sin t)^{N-1}  dt\right] \qq(dq) \nonumber \\
&=&   \liminf_{\ve\downarrow 0}  \lambda_{\bar{v}} \frac{\mm_Y(Q_i^\ve)- \mm_Y(Q_i)}{\ve} =    \lambda_{\bar{v}} \, \mm_Y^+(Q_i), \; i=1,2,\label{eq:mYQi}
\end{eqnarray}
where we set $\lambda_{\bar{v}}:= c \int_0^{ \min\{r_{\bar{v}}, \pi-r_{\bar{v}}\}} (\sin t)^{N-1}  dt$. Notice that $\lambda_{\bar{v}}>0$ for $\bar{v}\in (0,1)$. Recalling that, from Step 1, $(Y,\sfd_Y,\mm_Y)$ is an  $\RCD^*(N-2,N-1)$ space,   from Theorem \ref{T:iso} it follows that
\begin{equation}\label{eq:mYQi2}
\mm_{Y}^{+}(Q_{1}) \geq \cI_{N-2,N-1,\pi} (\mm_{Y}(Q_{1})) > 0, \quad  \mm_{Y}^{+}(Q_{2}) \geq \cI_{N-2,N-1,\pi} (\mm_{Y}(Q_{2})) > 0.
\end{equation}
Since by construction  $\mm_Y(Q_1)+\mm_Y(Q_2)=1$, the combination of  \eqref{eq:Ave-A},  \eqref{eq:mYQ1}, \eqref{eq:mYQ2},   \eqref{eq:mYQi} and \eqref{eq:mYQi2} yields
$$
\mm^+(A):=\liminf_{\ve\downarrow 0} \frac{\mm(A^\ve)-\mm(A)}{\ve}> \cI_{N-1,N,\pi}(\bar{v}), 
$$
contradicting that  $\mm^+(A)= \cI_{(X,\sfd,\mm)}(\bar{v})=\cI_{N-1,N,\pi}(\bar{v})$.

%
%
%
%

\medskip

{\textbf{Step 3.}} \\
We claim that if  $(X,\sfd,\mm)=[0,\pi]\times^{N-1}_{\sin}Y$ for some m.m.s.  $(Y,\sfd_1,\mm_1)$ with $\mm(Y)=1$ then, calling
$$
A=\{(t,q) \in  [0,\pi]\times^{N-1}_{\sin}Y \, : \, t\in [0,r_{\bar{v}}]\}
$$
where $r_{\bar{v}}$ is such that $\mu_{N-1,N}([0,r_{\bar{v}}])={\bar{v}}$, we have
$$
\mm(A)={\bar{v}} \quad \text{ and } \quad \mm^+(A)=\cI_{N-1,N,\pi}({\bar{v}}).
$$
The fact that $\mm(A)={\bar{v}}$ is clear by Fubini's Theorem, so let us show the second statement. 
For every $\ve>0$ observe that  the geometry of $A$ implies that   
\begin{eqnarray}
\frac{\mm(A^\ve)-\mm(A)}{\ve} &=&  \frac{1}{\ve} \;  \int_{Y}  \mu_{N-1,N}(\{t\in [0,\pi]\,: \, (t,q)\in A^\ve\setminus A\}) \, \mm_Y(dq)  \nonumber \\
&=&  \frac{\mu_{N-1,N}([0,r_{\bar{v}}+\ve])- \mu_{N-1,N}([0,r_{\bar{v}}]) }{\ve}.  \label{eq:mmmu}
\end{eqnarray}
Now observe  that 
\begin{eqnarray}
\lim_{\ve \downarrow 0}  \frac{\mu_{N-1,N}([0,r_{\bar{v}}+\ve])- \mu_{N-1,N}([0,r_{\bar{v}}]) }{\ve}&=& \liminf_{\ve \downarrow 0}  \frac{\mu_{N-1,N}([0,r_{\bar{v}}+\ve])- \mu_{N-1,N}([0,r_{\bar{v}}]) }{\ve}     \nonumber \\
&=& \mu_{N-1,N}^+([0,r_{\bar{v}}])=\cI_{N-1,N,\pi}({\bar{v}}).   \nonumber 
\end{eqnarray}
Therefore, taking a sequence $\ve_i \downarrow 0$ such that 
$$
\mm^+(A)=\liminf_{\ve \downarrow 0} \frac{\mm(A^\ve)-\mm(A)}{\ve}=\lim_{i\to \infty}  \frac{\mm(A^{\ve_i})-\mm(A)}{\ve_i},
$$
we can pass to the limit in \eqref{eq:mmmu} over the sequence $\ve_i\downarrow 0$ and conclude the proof. 
\hfill$\Box$

\subsection{Almost equality in L\'evy-Gromov  implies  almost rigidity} 

Let us start by the following  lemma.

\begin{lemma}\label{lem:Icont}
For every $N>1$ and every $v \in [0,1]$, the map 
$$[0, N-1) \ni \delta \mapsto \cI_{N-1-\delta, N-\delta, \infty}(v) \in \R^+  \quad \text{is continuous}.$$
In particular for every  $\eta>0$ there exists $\bar{\delta}=\bar{\delta}(N,\eta)>0$ such that
$$\left| \cI_{N-1,N,\infty} (v) - \cI_{N-1-\delta, N+\delta, \infty} (v)   \right| \leq \eta , \quad \forall v \in [0,1], \; \forall \delta\in [0, \bar{\delta}]. $$
\end{lemma} 

\begin{proof}
By \cite{Mil} we know that, called   
$$\mu_{N-1-\delta,N+\delta,\infty} :=c_{N,\delta} \left[\sin\left(  \sqrt{\frac{N-1-\delta}{N+\delta-1}} \, t \right)\right]^{N+\delta-1} \chi_{\left[0, \sqrt{ \frac{N+\delta-1}{N-1-\delta}} \pi \right] } (t) \,{\mathcal L}^1(dt)  \in {\mathcal F}_{N-1-\delta,N+\delta,\infty},$$
where $c_{N,\delta}>0$ is the normalizing constant, there exists   $A=A_{\delta,N,v}$ of the form $(-\infty, a_{\delta, N,v})$ such that  for every $v \in [0,1]$
$$
\mu_{N-1-\delta,N+\delta,\infty}(A)= v \quad \text{and} \quad \mu^{+}_{N-1-\delta,N+\delta,\infty}(A)= \cI_{N-1-\delta,N+\delta,\infty}(v).
$$
It is straightforward to check that the maps
$$
 \delta \mapsto  \left[\sin\left(  \sqrt{\frac{N-1-\delta}{N+\delta-1}} (\cdot) \right)\right]^{N+\delta-1} \chi_{\left[0, \sqrt{ \frac{N+\delta-1}{N-1-\delta}} \pi \right] } (\cdot) \in C(\R,\|\cdot\|_\infty), \quad \delta \mapsto a_{\delta,N,v} \in \R^+, \quad  \delta \mapsto c_{N,\delta} \in \R^+
$$
are continuous. Since  by the  Fundamental Theorem of Calculus
$$\mu^{+}_{N-1-\delta,N+\delta,\infty}(A_{\delta,N,v}) = c_{N,\delta}   \left[\sin\left(  \sqrt{\frac{N-1-\delta}{N+\delta-1}} \, a_{\delta,N,v} \right)\right]^{N+\delta-1},$$
the claim follows.
\end{proof}

We can now prove the almost rigidity theorem.

\medskip
\noindent
{\textbf{Proof of Theorem \ref{thm:AlmRig}}}.
Let $\eta=\eta(v,N,\ve_0)>0$ be given by Lemma \ref{lem:D<inf} so that  for every $\delta\in \left[0, \frac{N-1}{2}\right]$ and every $D\in(0,\pi-\ve_0)$ it holds
\begin{equation}\label{eq:D<infeta}
\cI_{N-1-\delta, N+\delta,D}(v) \geq \cI_{N-1-\delta, N+\delta,\infty}(v)+3\eta.
\end{equation}
Moreover Lemma \ref{lem:Icont} ensures that for $\delta>0$ small enough it holds
\begin{equation} \label{eq:I>eta}
\cI_{N-1-\delta, N+\delta, \infty} (v) \geq  \cI_{N-1,N,\infty} (v) - \eta.
\end{equation}
Assume by contradiction there exists $\ve_0>0$ such that for every $\delta>0$ there is an $\RCD^*(N-1-\delta,N+\delta)$ space $(X,\sfd,\mm)$ containing a Borel subset $A\subset X$ satisfying 
\begin{equation}\label{eq:Aeta}
\mm(A)=v \quad \text{and} \quad \mm^+(A)\leq \cI_{N-1, N, \infty}(v)+\eta    
\end{equation}
and  such that $\diam((X,\sfd))\leq \pi-\ve_0$.

If we argue analogously to the  proof of Theorem \ref{thm:Rigidity} using   \eqref{eq:D<infeta}, \eqref{eq:I>eta} and \eqref{eq:Aeta}, we reach the contradiction
\begin{eqnarray}
\cI_{N-1-\delta,N+\delta,\infty}(v)+2\eta &\geq & \cI_{N-1,N,\infty} (v) + \eta  \geq  \mm^+(A) \geq  \int_{Q} \left( (h_{q}\mathcal{L}^1)^+(g(q,\cdot)^{-1}(A))  \right)\, \qq(dq)   \nonumber \\
							    &\geq &  \int_{Q}  \cI_{N-1-\delta,N+\delta, |\supp(h_q)|} (\bar{v}) \, \qq(dq)    \geq   \cI_{N-1-\delta,N+\delta,\infty}(\bar{v})+3\eta. \nonumber
							    \end{eqnarray}							    
\hfill$\Box$

Corollary  \ref{cor:AlmRig} is   a straightforward consequence  of  Theorem \ref{thm:AlmRig} combined with  the Maximal Diameter Theorem \cite{Ket} and  the compactness/stability of $\RCD^*(K,N)$ spaces with respect to the  mGH convergence.  Let us briefly outline the arguments for completeness.

\medskip
\noindent
{\textbf{Proof of Corollary \ref{cor:AlmRig}}}.
Fix $N\in [2, \infty) $, $v \in (0,1)$ and assume by contradiction there exist $\ve_0>0$ and  a sequence $(X_j, \sfd_j, \mm_j)$ of $\RCD^*(N-1-\frac{1}{j}, N+\frac{1}{j})$ spaces such that  $\cI_{(X_j,\sfd_j,\mm_j)}(v)\leq \cI_{N-1,N,\infty}(v)+\frac{1}{j}$ but
\begin{equation}\label{eq:contrj}
\sfd_{mGH}(X_j, [0,\pi] \times_{\sin}^{N-1} Y) \geq \ve_0 \quad \text{for every $j\in \N$}
\end{equation}
and every  $\RCD^*(N-2,N-1)$ space $(Y, \sfd_Y, \mm_Y)$ with $\mm_Y(Y)=1$.   Observe that Theorem  \ref{thm:AlmRig} yields 
\begin{equation}\label{eq:diam}
\diam((X_j, \sfd_j))\to \pi.
\end{equation}
 By the compactness/stability property of $\RCD^*(K,N)$ spaces recalled in Theorem \ref{thm:CompRCD}  we get  that, up to subsequences, the spaces $X_j$ mGH-converge to a limit $\RCD^*(N-1,N)$ space $(X_\infty, \sfd_\infty, \mm_\infty)$.   Since the diameter is continuous under mGH convergence of uniformly bounded spaces,    \eqref{eq:diam} implies  that $\diam((X_\infty, \sfd_\infty))=\pi$.  But then by the Maximal Diameter Theorem \cite{Ket} we get that $(X_\infty, \sfd_\infty, \mm_\infty)$ is isomorphic to a spherical suspension  $[0,\pi] \times_{\sin}^{N-1} Y$ for some  $\RCD^*(N-2,N-1)$ space $(Y, \sfd_Y, \mm_Y)$ with $\mm_Y(Y)=1$.  Clearly this contradicts \eqref{eq:contrj} and the thesis follows.
 \hfill$\Box$


\begin{thebibliography}{10}




 
%
%
%
%


%
%
%
%
%
%
%
%
%
%
%
%
%
%
%
%
%
%
%
%
%
%
%
%
%
%


%
%
%
%
%
%
%
%
%
%
%
%
%
%

\bibitem{AFP}   L.~Ambrosio, N.~Fusco and D.~Pallara,
\newblock Functions of Bounded Variation and Free Discontinuity Problems,
\newblock{\em Oxford Mathematical Monographs}, (2000).


\bibitem{AGMR12} L.~Ambrosio, N.~Gigli, A.~Mondino, and T.~Rajala,
\newblock Riemannian {R}icci curvature lower bounds in metric measure spaces with $\sigma$-finite measure,
\newblock{\em Trans. Amer. Math. Soc.,}   \textbf{367},  7,  (2015), 4661--4701.


\bibitem{AGS} L.~Ambrosio, N~Gigli, G.~Savar\'e,
\newblock Bakry-\'Emery curvature-dimension condition and Riemannian Ricci curvature bounds.
\newblock{\em  Ann. Probab.,}  \textbf{43},  1,  (2015), 339--404. 

\bibitem{AGS11a}
\leavevmode\vrule height 2pt depth -1.6pt width 23pt,
\newblock Calculus and heat flow in  metric measure spaces and applications to spaces with {R}icci bounds from
  below, 
\newblock{\em   Invent. Math.,}   \textbf{195},  2,  (2014), 289--391.

\bibitem{AGS11b} \leavevmode\vrule height 2pt depth -1.6pt width 23pt,
 \newblock Metric measure spaces with {R}iemannian {R}icci curvature bounded from below, 
 \newblock{\em Duke Math. J.},  \textbf{163}, (2014), 1405--1490. 



%

%


\bibitem{AMS} L.~Ambrosio, A.~Mondino and G.~Savar\'e,  
\newblock Nonlinear diffusion equations and curvature conditions in metric measure spaces, 
\newblock {\em Preprint  arXiv:1509.07273}.

\bibitem{AMSLocToGlob}
\leavevmode\vrule height 2pt depth -1.6pt width 23pt,
\newblock{ On the Bakry-\'Emery condition, the gradient estimates and the Local-to-Global property of $RCD^*(K,N)$ metric measure spaces},
\newblock{\em  J. Geom.  Anal.,}
 \textbf{26}, (2016), 24--56.

%
%
%

\bibitem{Bayle} V.~Bayle,
\newblock A Differential Inequality for the Isoperimetric Profile,
\newblock {\em Int. Math. Res. Not.}, (2004),  \textbf{7},  311--342.

\bibitem{BS10}  K.~Bacher and K.-T. Sturm,
\newblock Localization and tensorization properties of the curvature-dimension condition for metric measure spaces,
\newblock { \em J. Funct. Anal.,}  \textbf{259}, (2010),  28--56.


\bibitem{BakryEmery_diffusions}  D.~Bakry and M.~Emery,
\newblock Diffusions hypercontractives
\newblock{\em Seminaire de Probabilites XIX, Lecture Notes in Math.}, Springer-Verlag, New York. 1123 (1985), 177--206.


\bibitem{BakryLedoux}  D.~Bakry and M.~Ledoux,
\newblock{A logarithmic Sobolev form of the Li-Yau parabolic inequality},
\newblock{\em Rev. Mat. Iberoam.,} \textbf{22}, 2, (2006), 683--702. 

\bibitem{BurZal} Y.D.~Burago and V.A~Zalgaller,
\newblock{Geometric inequalities},
\newblock{\em  Grundlehren der Mathematischen Wissenschaften [Fundamental Principles of Mathematical Sciences]}, \textbf{285}. Springer, Berlin (1988). 

%
%
%
%

\bibitem{biacava:streconv}
S.~Bianchini and F.~Cavalletti,
\newblock The {M}onge problem for distance cost in geodesic spaces.
\newblock {\em Commun. Math. Phys.},  \textbf{318},  (2013), 615 -- 673.



\bibitem{cava:MongeRCD}  F.~Cavalletti,
\newblock  Monge problem in metric measure spaces with Riemannian curvature-dimension condition,
\newblock{\em Nonlinear Anal.}, \textbf{99}, (2014), 136--151.


\bibitem{cava:decomposition} \leavevmode\vrule height 2pt depth -1.6pt width 23pt,
\newblock Decomposition of geodesics in the {W}asserstein space and the  globalization property.
\newblock {\em Geom. Funct. Anal.}, \textbf{24}, (2014),  493 -- 551.



\bibitem{CM2} F.~Cavalletti and A.~Mondino.
\newblock Sharp geometric and functional inequalities in metric measure spaces with lower Ricci curvature bounds.
\newblock {\em preprint arXiv:1505.02061},  to appear in Geometry \& Topology. 



\bibitem{CM3} F.~Cavalletti and A.~Mondino.
\newblock Optimal maps in essentially non-branching spaces.
\newblock {\em preprint  arXiv:1609.00782}, to appear in Commun. Contemp. Math.



\bibitem{cavasturm:MCP} F.~Cavalletti and K.-T.~Sturm.
\newblock Local curvature-dimension condition implies measure-contraction  property.
\newblock {\em J. Funct. Anal.}, \textbf{262}, 5110 -- 5127, 2012.


\bibitem{CC96} G.~Cheeger and T.H.~Colding,
\newblock  Lower bounds on Ricci curvature and the almost rigidity of warped products,
\newblock {\em Annals of Math.}, \textbf{144}, 1,  (1996),  189--237.


\bibitem{CC1}
\leavevmode\vrule height 2pt depth -1.6pt width 23pt,
\newblock On the structure of spaces with {R}icci curvature bounded below. I.
\newblock {\em J. Diff. Geom.}, \textbf{45} (1997),  406 -- 480.


\bibitem{CC2}
\leavevmode\vrule height 2pt depth -1.6pt width 23pt,
\newblock On the structure of spaces with {R}icci curvature bounded below. II.
\newblock {\em J. Diff. Geom.},  \textbf{54},  (2000), 13--35.  


\bibitem{CC3}
\leavevmode\vrule height 2pt depth -1.6pt width 23pt,
\newblock On the structure of spaces with {R}icci curvature bounded below. III.
\newblock {\em J. Diff. Geom.}, \textbf{54}, (2000), 37 -- 74.

\bibitem{CL} M.~Cicalese, G.P.~Leonardi,
\newblock{  A selection principle for the sharp quantitative isoperimetric inequality},
\newblock{\em Arch. Rat. Mech. Anal.,} \textbf{206}, 2,  (2012),  617--643.  

 \bibitem{CN} T.~Colding, A.~Naber, 
 \newblock Sharp H\"older continuity of tangent cones for spaces with a lower Ricci curvature bound and applications
\newblock {\em Annals of Math.}, \textbf{176}, (2012).
 

\bibitem{Croke} C.B.~Croke,
\newblock An eigenvalue pinching theorem, 
\newblock {\em Invent. Math.}, \textbf{68}, 2,  (1982),  253--256.



\bibitem{EiMe}  M.~Eichmair and J.~Metzger,
\newblock  Unique isoperimetric foliations of asymptotically flat manifolds in all dimensions,
\newblock {\em Invent. Math.},  \textbf{194}, (2013), 591--630. 


\bibitem{EKS} M~Erbar, ~Kuwada and K.T.~Sturm,
\newblock On the Equivalence of the Entropic Curvature-Dimension Condition and Bochner's Inequality on Metric Measure Space,
\newblock {\em  Invent. Math.},  \textbf{201},  3,  (2015), 993--1071.  
%



\bibitem{Evans:PDEs} L.~C. Evans,
\newblock {\em Partial Differential Equations}, 
\newblock Graduate Studies in Mathematics, vol. 19, AMS, 1998.


\bibitem{FiMP} A.~Figalli, F.~Maggi and A.~Pratelli,
\newblock{A mass transportation approach to quantitative isoperimetric inequalities},
\newblock{\em Invent. Math.}, \textbf{182}, 1, (2010), 167--211.

\bibitem{Fre:measuretheory4} D.~H. Fremlin,
\newblock {\em Measure Theory}, volume~4.
\newblock Torres Fremlin, (2002).

\bibitem{FuMP} N.~Fusco, F.~Maggi and A.~Pratelli,
\newblock{The sharp quantitative isoperimetric inequality},
\newblock{\em Annals of Math.}, \textbf{168}, (2008), 941--980.


\bibitem{BBG} P.H.~B\'erard, G.~Besson, and S.~Gallot,
\newblock{Sur une in\'egalit\'e isop\'erim\'etrique qui g\'en\'eralise celle de Paul L\'evy-Gromov [An isoperimetric inequality generalizing the Paul Levy-Gromov inequality],}
\newblock{\em Invent. Math.}, \textbf{80}, 2, (1985), 295--308 (French).

\bibitem{GaMo}   N.~Garofalo and A.~Mondino, 
\newblock{Li-Yau and Harnack type inequalities in $\RCD^*(K,N)$ metric measure spaces}, 
\newblock{\em Nonlinear Analysis: Theory, Methods \& Applications}, \textbf{95}, (2014), 721--734.


\bibitem{GigliMap}  N.~Gigli,
\newblock Optimal maps in non branching spaces with Ricci curvature bounded from below,
\newblock{\em Geom. Funct. Anal.}, \textbf{22} (2012) no. 4, 990--999.

\bibitem{GigliSplitting}
\leavevmode\vrule height 2pt depth -1.6pt width 23pt,
\newblock {The splitting theorem in non-smooth context},
\newblock{\em preprint arXiv:1302.5555}, (2013).


\bibitem{GMS2013}  N.~Gigli, A.~Mondino and G.~Savar\'e,
\newblock {Convergence of pointed non-compact metric measure spaces and stability of Ricci curvature bounds and heat flows},
\newblock{\em Proc. London Math. Soc.,}   \textbf{111},  (5),  (2015), 1071--1129.
 

\bibitem{GMR2013}  N.~Gigli, A.~Mondino and T.~Rajala,
\newblock{Euclidean spaces as weak tangents of infinitesimally Hilbertian metric measure spaces with Ricci curvature bounded below}
\newblock{\em Journal fur die Reine und Ang. Math.}, \textbf{705},  (2015), 233--244.


\bibitem{GRS2013}  N.~Gigli, T.~Rajala and K.T.~Sturm,
\newblock{Optimal maps and exponentiation on finite dimensional spaces with Ricci curvature bounded from below}
\newblock{\em 	preprint  arXiv:1305.4849}, to appear in J. Geom. Analysis. 



\bibitem{Gro}  M.~Gromov,
\newblock Metric structures for Riemannian and non Riemannian spaces,
{\em  Modern Birkh\"auser Classics}, (2007).

\bibitem{GrMi} M.~Gromov and V.~Milman,
\newblock{Generalization of the spherical isoperimetric inequality to uniformly convex Banach spaces.}
\newblock{\em Compositio Math.}, \textbf{62},  3, (1987), 263--282.

\bibitem{Honda}  S.~Honda,
\newblock Cheeger constant, $p$-Laplacian, and Gromov-Hausdorff convergence,
\newblock {\em preprint} (2014) arXiv:1310.0304v3.


\bibitem{KaLoSi} R.~Kannan, L.~Lov\'asz and M.~Simonovits,
\newblock{Isoperimetric problems for convex bodies and a localization lemma},
\newblock{\em Discrete Comput. Geom.,} \textbf{13},  3-4, (1995), 541--559.

\bibitem{Ket} C.~Ketterer,
\newblock Cones over metric measure spaces and the maximal diameter theorem.
\newblock{\em J. Math. Pures Appl.} \textbf{103}, 5,  (2015),  1228--1275.


\bibitem{klartag} B.~Klartag,
\newblock Needle decomposition in Riemannian geometry, 
\newblock{\em preprint  arXiv:1408.6322}, to appear in Mem.  AMS.

\bibitem{lottvillani:metric} J.~Lott and C.~Villani,
\newblock Ricci curvature for metric-measure spaces via optimal transport,
\newblock{\em Ann. of Math.} (2) \textbf{169} (2009), 903--991.

\bibitem{LoSi} L.~Lov\'asz and M.~Simonovits,
\newblock{Random walks in a convex body and an improved volume algorithm,}
\newblock{\em Random Structures Algorithms,} \textbf{4},  4, (1993), 359--412.

\bibitem{Mag} F.~Maggi,
\newblock Sets of Finite Perimeter and Geometric Variational Problems  (an Introduction to Geometric Measure Theory),
\newblock{\em Cambridge Studies in Advanced Mathematics}, (2012).

\bibitem{Mil} E.~Milman,
\newblock Sharp Isoperimetric Inequalities and Model Spaces for Curvature-Dimension-Diameter Condition,
\newblock{\em J. Europ. Math. Soc.},  \textbf{17}, (5),  (2015), 1041--1078.

\bibitem{MilRot} E.~Milman and L.~Rotem,
\newblock Complemented Brunn-Minkowski Inequalities and Isoperimetry for Homogeneous and Non-Homogeneous Measures,
\newblock{\em Advances in Math.}, \textbf{262}, 867--908, (2014). 

\bibitem{MN} A.~Mondino and A.~Naber, 
\newblock Structure Theory of Metric-Measure Spaces with Lower Ricci Curvature Bounds I,
\newblock{\em preprint arXiv:1405.2222}.



\bibitem{Mor} F.~Morgan,
\newblock Geometric Measure Theory (A Beginner's Guide), 
\newblock{\em Elsevier/Academic Press,}
Amsterdam, Fourth edition, (2009).

 
\bibitem{MorPol} \leavevmode\vrule height 2pt depth -1.6pt width 23pt,
\newblock  In polytopes, small balls about  some vertex minimize perimeter. 
\newblock{\em J. Geom. Anal.} \textbf{17},  97--106, 
(2007).

\bibitem{MR} F.~Morgan and M.~Ritor\'e,
\newblock Isoperimetric regions in cones,
\newblock{\em Trans. Amer. Math. Soc.}, \textbf{354},  (2002), 2327--2339.


\bibitem{Ohta}  S.I..~Ohta,
\newblock  Finsler interpolation inequalities,
\newblock  {\em Calc. Var. Partial Differential Equations}, \textbf{36}, (2009), 211--249.

\bibitem{Oss} R.~Osserman, 
\newblock The isoperimetric inequality,
\newblock{\em  Bull. Am. Math. Soc.}, \textbf{84} (6), 1182--1238,
(1978)



\bibitem{PW} L.E.~Payne and H.F.~Weinberger,
\newblock{ An optimal Poincar\'e inequality for convex domains},
\newblock{\em  Arch. Rational Mech. Anal.}, \textbf{5}, (1960), 286--292.

\bibitem{Pet} A. Petrunin,
\newblock Harmonic functions on Alexandrov spaces and their applications,
\newblock{\em  Electronic Res. Announc. AMS}, \textbf{9}, (2003), 135---141. 

\bibitem{PLSV} \leavevmode\vrule height 2pt depth -1.6pt width 23pt,
\newblock Alexandrov meets Lott-Sturm-Villani,
\newblock{\em M\"unster J. Math.}, \textbf{4}, (2011), 53--64.

\bibitem{R2011} T.~Rajala,
\newblock {Local Poincar\'e inequalities from stable curvature conditions on metric spaces},
\newblock{\em Calc. Var. Partial Differential Equations,} \textbf{44},  (2012), 477--494.
 
\bibitem{RS2014} T.~Rajala and K.T.~Sturm,
\newblock {Non-branching geodesics and optimal maps in strong $\CD(K,\infty)$-spaces},
\newblock{\em Calc. Var. Partial Differential Equations,} \textbf{50},  (2014), 831--846.
 


\bibitem{Rit} M.~Ritor\'e, 
\newblock Geometric flows, isoperimetric inequalities and hyperbolic geometry, Mean curvature flow and isoperimetric inequalities,
\newblock {\em Adv. Courses Math. CRM Barcelona},  45-113, Birkh\"auser, Basel, (2010).


\bibitem{Ros} A.~Ros,
\newblock The isoperimetric problem,
\newblock{Lecture series at the Clay Mathematics Institute, Summer School on the Global Theory of Minimal Surfaces}, MSRI, Berkeley, California, (2001). 


\bibitem{Savare13} G.~Savar\'e, 
\newblock{Self-improvement of the Bakry-\'Emery condition and Wasserstein contraction of the heat flow in RCD$(K,\infty)$ metric measure spaces},
\newblock{\em Disc. Cont. Dyn. Sist. A},  \textbf{34}, (2014), 1641--1661.   


\bibitem{sturm:I} K.T.~Sturm, 
\newblock On the geometry of metric measure spaces. I,
\newblock{\em Acta Math.} \textbf{196} (2006), 65--131.

\bibitem{sturm:II} K.T.~Sturm, 
\newblock On the geometry of metric measure spaces. II,
\newblock{\em Acta Math.} \textbf{196} (2006), 133--177.

\bibitem{Vil} C.~Villani, 
\newblock Optimal transport. Old and new, 
\newblock{\em Grundlehren der Mathematischen Wissenschaften}, \textbf{338}, Springer-Verlag, Berlin, (2009).


\bibitem{zhangzhu} H.C.~Zhang and X.P.~Zhu,
\newblock Ricci curvature on Alexandrov spaces and rigidity theorems, 
\newblock{\em Comm. Anal. Geom.} \textbf{18}, (3), (2010), 503--553.



\end{thebibliography}
\end{document}